\documentclass[mnsc,nonblindrev]{informs3} 

\OneAndAHalfSpacedXI
\usepackage{natbib}
\bibpunct[, ]{(}{)}{,}{a}{}{,}%
\def\bibfont{\small}%
\usepackage{booktabs} 

\usepackage[ruled]{algorithm2e} 

\usepackage{ bbm, color, enumerate, amsbsy, amsfonts, amsopn, amssymb,  amsxtra, bezier, graphicx, latexsym, verbatim,
	pictexwd, supertabular, url, dsfont,leftidx, setspace}
\usepackage{mathtools,bm}
\usepackage{multirow}
\usepackage{footnote}
\makesavenoteenv{tabular}
\makesavenoteenv{table}

\newcommand{\mb}[1]{\ensuremath{\boldsymbol{#1}}}

\def\opt{\textsc{OPT}}
\def\alg{\textsc{ALG}}

\DeclarePairedDelimiter{\ceil}{\lceil}{\rceil}
\DeclarePairedDelimiter{\floor}{\lfloor}{\rfloor}

\SetAlFnt{\small}
\SetAlCapFnt{\small}
\SetAlCapNameFnt{\small}
\SetAlCapHSkip{0pt}
\IncMargin{-\parindent}

\TheoremsNumberedThrough     
\ECRepeatTheorems

\EquationsNumberedThrough

\begin{document}
	
	\TITLE{Submodular Order Functions and Assortment Optimization} 
\ARTICLEAUTHORS{	\AUTHOR{	Rajan Udwani}
\AFF{UC Berkeley, IEOR, 
\EMAIL{rudwani@berkeley.edu}}}

\ABSTRACT{%
We define a new class of set functions that in addition to being monotone and subadditive, also admit a very limited form of submodularity defined over a permutation of the ground set. We refer to this permutation as a submodular order. This class of functions includes monotone submodular functions as a sub-family. We give fast algorithms with strong approximation guarantees for maximizing submodular order functions under a variety of constraints and show a nearly tight upper bound on the highest approximation guarantee achievable by algorithms with polynomial query complexity.  Applying this new notion to the problem of constrained assortment optimization in  fundamental choice models, we obtain new algorithms that are both faster and have stronger approximation guarantees (in some cases, first algorithm with constant factor guarantee). We also show an intriguing connection to the maximization of monotone submodular functions in the streaming model, where we recover best known approximation guarantees as a corollary of our results.
}%

\KEYWORDS{Submodular order, Assortment Optimization, Approximation Algorithms}
\maketitle

\section{Introduction}
{\color{black} Given a ground set $N$, consider the problem of finding a (feasible) subset $S$ that maximizes a set function $f:2^N\to \mathbb{R}$ that satisfies the following properties,
\begin{eqnarray*}
&\text{\emph{Monotonicity:} }	&f(A)\,\geq\, f(B)\quad \forall B\subseteq A,\\
&\text{\emph{Submodularity:} } &f(A\cup \{e\})-f(A)\leq f(B\cup\{e\})-f(B)\quad \forall B\subseteq A\subseteq N,\, e\in N\backslash A.
\end{eqnarray*}
Since the work of \cite{nemhauser1978analysis}, this problem has received significant and continued interest in Operations Research, Computer Science, and many other communities due to a variety of modern applications, 
such as, feature selection \citep{das2011submodular, wei2015submodularity}, sensor placement \citep{krause2008near,leskovec2007cost}, influence maximization \citep{krause2008robust}, data summarization \citep{cardstream}, and many others. 
Despite this richness of submodular functions, there are important subset selection problems where the objective ($f$) is \emph{not} submodular. 

Of particular interest to us is the setting of \emph{assortment optimization} where the objective is not submodular. The problem arises widely in many industries including retailing and online
advertising, where given a universe of substitutable products, the seller wishes to select a subset of products (called an assortment) with maximum expected revenue. The effect of substitution between products is captured by a \emph{choice model} 
$\phi: N\times 2^N\to [0,1]$. 
Given assortment $S$, the probability that customer chooses product $i\in S$ is given by $\phi(i,S)$. Customer may choose an \emph{outside option} that is not in $S$ with probability $1-\sum_{i\in S}\phi(i,S)$. Given (fixed) prices $(r_i)_{i\in N}$, the expected revenue is,
\[R_{\phi}(S)=\sum_{i\in S}r_i \phi(i,S).\]
Given the choice model $\phi$, we are interested in the problem of finding the revenue optimizing assortment subject to certain types of constraints that we will discuss shortly. 
For a general choice model,  \citet{aouad2018approximability} show that even the unconstrained assortment optimization problem is hard to approximate. The optimization problem is more tractable for commonly used simple choice models such as the Multinomial Logit (MNL) choice model (defined in Section \ref{sec:backasst}) but the revenue objective is not submodular even in this special case. The objective is also non-monotone but one can transform it into a monotone objective (see Section \ref{sec:asstresult}). 
Given this intractability and lack of general structure, the field of assortment optimization has evolved through algorithms that are tailored for individual families of choice models. 

Motivated by this, we are interested in finding a unifying structural property (for set functions) that leads to computationally tractable optimization problems which capture both monotone submodular function maximization and assortment optimization problems as special cases. As a starting point, consider monotone functions that satisfy the following mild condition, 
\begin{eqnarray*}
\text{\emph{Subadditivity:} }	f(A)+f(B)\,\geq\, f(A\cup B)\quad \forall A,B\subseteq N.
\end{eqnarray*}
Every non-negative submodular function is subadditive and the revenue function in assortment optimization is also subadditive for a general family of random utility based choice models \citep{kok2008assortment,berbeglia2020assortment}.
Using the (standard) normalization $f(\emptyset)=0$, monotonicity implies that $f$ is non-negative. 
Given a monotone subadditive function $f$, we are interested in solving,
\[\argmax_{S\in\mathcal{F}} f(S),\]
where $\mathcal{F}$ corresponds to one of the following types of constraints that commonly arise in applications, 
\begin{eqnarray*}
&\text{Cardinality ($k$)     }\quad	&\{S\mid   |S|\leq k\}, \\
&\text{Budget/Knapsack $\left(\{b_i\}_{i\in{N}},B\right)$     }\quad &\{S\mid b(S)\leq B\},\\ 
&\text{Matroid ($\mathcal{I}$)      }\quad &\{S\mid S\in \mathcal{I}\},	
\end{eqnarray*}
here $b(S)=\sum_{i\in S} b_i$ and $\mathcal{I}$ is the family of independent sets of a matroid. Note that both budget and matroid constraint include cardinality constraint as a special case. In the absence of a constraint the ground set $N$ maximizes function value (due to monotonicity). However, constrained optimization is 
computationally intractable for monotone subadditive functions even with a cardinality constraint.}
\begin{theorem}[Adapted from Theorem 6.1 in \cite{mor}]\label{mor}
Any algorithm that makes polynomial number of queries (in size $n$ of ground set), cannot have approximation guarantee better than $n^{-\Omega(1)}$ for maximizing a {\color{black}monotone subadditive function} subject to cardinality constraint.
\end{theorem} 
%
%

In light of this impossibility result, we introduce a new structural property, namely, a \emph{submodular order}, and show that it drastically alters the optimization landscape by allowing efficient algorithms with constant factor guarantee for all constraint families of interest. To define this new notion we introduce some notation. Given a permutation $\pi$ over $N$, we index elements in the order given by $\pi$. Let $r_{\pi}(S)$ denote the element in $S$ with the largest index and $l_{\pi}(S)$ the element with the smallest index. It helps to visualize the indexing as an increasing order from left to right; $r_{\pi}(S)$ is the rightmost element of $S$ and $l_{\pi}(S)$ is the leftmost. We define the marginal value function,
\[f(X|S)=f(X\cup S)-f(S).\]

\noindent \textbf{Submodular order:} A permutation $\pi$ of elements in $N$ is a submodular order if for all sets $B\subseteq A$ and $C$ to the right of $A$ i.e., $l_{\pi}(C)>r_{\pi}(A)$, we have
\begin{equation}\label{def}
f(C\mid A)\leq f(C\mid B).
\end{equation}
{\color{black} We refer to a function $f$ as a \emph{submodular order function} if there exists a submodular order for the function. We refer to function with submodular order $\pi$ simply as $\pi$-\emph{submodular ordered}.  Our main focus is on functions for which a submodular order is known but our algorithmic results extend to certain settings where a submodular order may not be known or may not even exist (discussed further in Section \ref{sec:framework}).}
\smallskip
%
%
%
%
%

\noindent \textbf{Comparison with submodular functions:}
For a submodular function $f$, inequality \eqref{def} holds for \emph{every} set $C$. 
In fact, it can be shown that a function is submodular if and only if every permutation of the ground set is a submodular order. 
From an optimization standpoint, there is a substantial difference between submodularity and submodular order. 
Consider the classic work of 
\citet{nemhauser1978analysis}, which showed that the greedy algorithm that iteratively builds a solution by adding an element with the largest marginal value, achieves a guarantee of $(1-1/e)$ for maximizing monotone submodular functions subject to cardinality constraint. 
The following example demonstrates that greedy can be arbitrarily bad even for a simple submodular order function.

\noindent \textbf{Example 1:} Consider a ground set $\{1,\cdots,2k+1\}$. Let singleton values $f(\{e\})$ equal 1 for all \emph{good} elements $e\in\{1,\cdots,k\}$, equal $\epsilon$ for all \emph{poor} elements $e\in\{k+1,\cdots,2k\}$ and finally, let $f(\{2k+1\})=1+\epsilon$. Let the first $2k$ elements be modular i.e., the value of a set $S=S_1\cup S_2$ with subset $S_1$ of good elements and $S_2$ of poor elements is simply $|S_1|+\epsilon|S_2|$. Finally, let $f(S_1\cup S_2\cup\{2k+1\})=\max\{|S_1|,1+\epsilon\}+\epsilon|S_2|$ for every subset $S_1$ of good elements and $S_2$ of poor elements. It can be verified that this function is monotone subadditive and the natural indexing $\{1,\cdots,2k+1\}$ is a submodular order. Notice that the set of all good elements has value $k$ and this is the optimal set of cardinality $k$ for every $k\geq 2$. However, the greedy algorithm would first pick element $2k+1$ and subsequently pick $k-1$ poor elements resulting in a total value of $1+\epsilon\, k$. For $\epsilon\to 0$, this is only a trivial $1/k$ approximation of the optimal value.

\emph{What about other algorithms for submodular maximization?} One may wonder if a greedy algorithm that chooses $m$ elements at each step for $m\geq 2$, admits better performance. The example above can be easily modified to show that this family of algorithms is arbitrarily bad in the worst case for any constant $m$ {\color{black}(see Appendix \ref{appx:multigrd})}. The more general schema of \emph{continuous} greedy algorithms (see 
\citet{calinescu2011maximizing}, \citet{feldman2011unified}), is  similarly ineffective. 

Local search is another well studied family of algorithms for submodular maximization (see \cite{feige2011maximizing}). These algorithms incrementally improve the value of a solution by swapping elements. For submodular order functions, local search fails to improve the value of a poor solution. 
In the example above, performing local search by swapping one element at a time (with the objective of improving function value) will not find any improvement on the set $\{k+1,\cdots,2k+1\}$. This set is a local maxima with value $1/k$ of the optimal. More generally, natural modifications of the example show that local search over $m$-tuples is arbitrarily bad in the worst case for any constant $m$ {\color{black}(see Appendix \ref{appx:multigrd})}. We now discuss our main algorithmic results for submodular order maximization and its applications.
\subsection{Our Contributions}
We start by describing our main results for submodular order maximization. The algorithms referenced in this section are presented subsequently in Section \ref{intuit}. 
All our approximation results hold more strongly under a milder version of submodular order, called \emph{weak} submodular order, defined as follows.  
\smallskip

\noindent \textbf{\emph{Weak} submodular order and $\pi$ nested sets:} Sets $B\subseteq A$ are $\pi$-nested if the left most element of $A\backslash B$ is to the right of $B$ i.e., $l_{\pi}(A\backslash B)>r_{\pi}(B)$. 
An order $\pi$ is a weak submodular order if $f(C\mid A)\leq f(C\mid B)$ for all $\pi$-nested sets $B\subseteq A$ and set $C$ with $l_{\pi}(C)>r_{\pi}(A)$.
\smallskip

The weak submodular order property does not impose any requirements on sets $A,B$ that are not $\pi$-nested. Since this is a milder condition the resulting family of functions is broader than (strong) submodular order functions. An algorithm with guarantee $\alpha$ for weak submodular order functions is also an $\alpha$ approximation for submodular order functions. Conversely, the upper bound on approximation guarantee for submodular order functions also applies to functions with weak submodular order. It is not meant to be obvious a priori but we establish that, in general, the stronger notion does not offer any benefits in terms of approximation guarantee.	
We establish the following approximation guarantees.  Table \ref{summary1} provides a summary.
\begin{theorem}\label{rescard}
For cardinality constrained maximization of a monotone subadditive function $f$ with a (known) weak submodular order, Algorithm \ref{calg} is $(1-\epsilon)\, 0.5$ approximate with $O(\frac{n}{\epsilon} \log k)$ oracle queries, for any choice of $\epsilon\in(0,1)$.
\end{theorem}
\begin{theorem}[Upper Bound]\label{impossible}
Given a monotone and subadditive function with (strong) submodular order on a ground set of size $n$, any algorithm for the cardinality constrained maximization problem that makes at most $poly(n)$ queries cannot have approximation guarantee better than $0.5+\epsilon$ for any $\epsilon>0$.
\end{theorem}
More generally, we have the following approximation results for budget and matroid constraint. Note that  corollary of Theorem \ref{impossible}, no efficient algorithm can have approximation guarantee strictly better than $0.5$ for maximization subject to these constraints. 
\begin{theorem}\label{resbudget}
For budget constrained maximization of a monotone subadditive function $f$ with a (known) weak submodular order,
\begin{enumerate}[(i)]
\item Algorithm \ref{bcalg} is $(1-\epsilon)/3$ approximate with query complexity $O(\frac{n}{\epsilon} \log n)$, $\forall \epsilon\in(0,1)$.
\item  Algorithm \ref{5balg} is $0.5-\epsilon$ approximate with query complexity $O(\frac{1}{\epsilon}n^{1+\frac{1}{\epsilon}} \log n)$, $\forall \epsilon\in(0,0.5)$.
\end{enumerate}
\end{theorem}
\begin{theorem}\label{resmatr}
Algorithm \ref{malg} is $0.25$ approximate with query complexity $O(nd)$ for the problem of maximizing a monotone subadditive function $f$ with a known (weak) submodular order, subject to a matroid constraint with matroid rank $d$. 
\end{theorem}
\noindent \emph{Remarks:} We prove Theorems \ref{rescard}, \ref{resbudget}, and \ref{resmatr} in Section \ref{sec:proofsubmod}. The proof of Theorem \ref{impossible} is given in Appendix \ref{sec:upb}. Given a noisy value oracle $\hat{f}$ for a submodular order function $f$, such that 
$(1-\delta)\, f(S)\leq \hat{f}(S)\leq (1+\delta)\, f(S)\,\, \forall S\subseteq N,$
we show that our approximation guarantees are reduced by an additional multiplicative factor of $\left(1-O(\frac{n\delta}{1-\delta})\right)$. 
This property is important in applications such as assortment optimization, where the value oracle is implemented by solving an optimization problem (see Section \ref{sec:asstresult}). 
We formalize this observation with the analysis of our algorithms. 
\begin{table}{}	 

\centering
\begin{tabular}{l|c|c|c|}
\multirow{2}{*}{}   & 
\multirow{2}{*}{Lower bound}
&\multirow{2}{*}{Upper bound}
&\multirow{2}{*}{\# Queries}     \\ 
&&&\\                                                     
\hline
\multirow{2}{*}{	Cardinality} &
\multirow{2}{*}{   $0.5 -\epsilon$}
&  \multirow{6}{*}{$0.5$}& \multirow{2}{*}{$O(\frac{n}{\epsilon} \log k)$ }\\
&&&\\
\cline{1-2}\cline{4-4}                            
\multirow{2}{*}{	Budget/Knapsack     }       & 
\multirow{2}{*}{ $0.5 -\epsilon $}
& & \multirow{2}{*}{$O(\frac{1}{\epsilon}n^{1+\frac{1}{\epsilon}} \log n)$}\\ 
&&&\\
\cline{1-2} \cline{4-4}
\multirow{2}{*}{	Matroid    }      & 
\multirow{2}{*}{0.25} & & \multirow{2}{*}{$O(nd)$}\\ 
&&&\\
\cline{1-4}                   \\
\end{tabular}
\caption{Results for constrained maximization of (weak) submodular order functions. $n$ denotes the size of the ground set, $k$ the cardinality parameter, and $d$ is the matroid rank. For budget constraint we also show a faster $1/3-\epsilon$ approximation. Approximation guarantees hold for weak (and strong) submodular order functions and the upper bounds hold for strong (and weak) submodular order functions.}
\label{summary1}
\end{table}	
\smallskip

\subsubsection*{Outline for rest of the paper.}
In Section \ref{sec:appl}, we present applications of submodular order function maximization. First, we discuss new results for assortment optimization. Then, in Section \ref{sec:introstream}, we discuss connections with streaming maximization of submodular functions. We present our algorithms and the underlying insights in Section \ref{intuit}. This is followed by analysis and proofs of our main results for submodular order maximization in Section \ref{sec:proofsubmod}, where we start with useful bounds for (weak) submodular order functions that drive the analysis of our algorithms. In Section \ref{sec:card}, we prove Theorem \ref{rescard} for the cardinality constrained problem. Section \ref{sec:budget} presents the analysis of Algorithms \ref{bcalg} and \ref{5balg} (Theorem \ref{resbudget}) for budget constraint. In Section \ref{sec:matr}, we establish Theorem \ref{resmatr} for matroid constraint. 
We give the proof of the upper bound (Theorem \ref{impossible}) for (strong) submodular order functions in Section \ref{sec:upb}.  Finally, Section \ref{sec:conclusion} concludes the discussion with several directions for further research. Proofs for main results on assortment optimization are included in Appendix \ref{sec:proofassort}. 

\section{Applications of Submodular Order Functions} \label{sec:appl}
In this section, we discuss application of submodular order function maximization to assortment optimization (Section \ref{sec:asstresult}) and to streaming maximization of submodular functions (Section \ref{sec:introstream}). 
\subsection{Constrained Assortment Optimization} \label{sec:backasst}
We start by providing background on assortment optimization and discussing related work. Recall that in assortment optimization, we assume that a choice
model $\phi: N\times 2^N\to [0,1]$ is given and we are interested in the problem of finding the revenue optimizing assortment subject to some constraint. Given assortment $S$, the probability that customer chooses product $i\in S$ is given by $\phi(i,S)$. Customer may choose an \emph{outside option} (not in $S$) with probability $1-\sum_{i\in S}\phi(i,S)$. Given (fixed) prices $(r_i)_{i\in N}$, the expected revenue is,
\[R_{\phi}(S)=\sum_{i\in S}r_i \phi(i,S).\]
Out of the many choice models that have been introduced in the literature, we are interested in the following well studied choice models. 

\noindent \textbf{Multinomial Logit Choice Model (MNL)} \citep{Bradley1952RankAO,Luce59,McFadden73, Plackett75}: This model is defined by parameters $v_i\geq 0$ for $i\in N$. $v_0\geq 0$ denotes the parameter for the outside option. The probability that a customer chooses product $i$ from assortment $S$ is proportional to $v_i$. Formally, 
\[\phi(i,S):= \frac{v_i}{v_0+\sum_{e\in S} v_e}.\]
\emph{Previous results for assortment optimization:} The assortment optimization problem under MNL choice is very well studied. Notice that the objective in this problem is non-monotone (and not submodular). 
\citet{talluri} showed that the unconstrained problem can be solved optimally and the optimal solution includes all products above a price threshold. 
\citet{rusmevichientong2010dynamic} gave a polynomial time algorithm for the cardinality constrained problem. \citet{davis2013assortment} and \citet{avadhanula2016tightness} showed that the MNL optimization problem can solved optimally in polynomial time under totally unimodular constraints (TUM).	 \citet{desir2014near} showed that the budget constrained version 
is NP hard and they gave a FPTAS for the problem. 

{\color{black}
\noindent \textbf{Mixture of MNL with Customization}~\citep{omar}:  Consider a population given by $m$ types of customers. Each customer chooses according to a MNL choice model that depends on their type. The choice model of the population is described by a Mixture of MNLs --  MMNL choice model \citep{mcfadden2000mixed}. Customer type is revealed on arrival and we offer a \emph{customized} assortment based on the type. An offered assortment can be any subset of the products that we keep in our selection and suppose that we can keep at most $k$ different products in the selection. Suppose that a random customer is type $j\in[m]$ with probability $\alpha_j$. Thus, we would like to select at most $k$ products to maximize,
\[\max_{S, |S|\leq k}\sum_{j\in[m]}\alpha_j \max_{X\subseteq S} R_{\phi_j}(X),\]
here $\phi_j$ is the MNL choice model for customer type $j$. A $\frac{1}{m}$ approximation for this problem can be obtained simply by solving a cardinality constrained MNL assortment problem for each type separately and picking the best of these solutions. \citet{omar} gave a substantially stronger $\frac{\Omega(1)}{\log m}$ approximation for this problem and showed that it is NP hard to approximate the optimal assortment better than $\left(1-\frac{1}{e}+\epsilon\right)$. For constant $m$, they gave a FPTAS. The unconstrained version of this problem (where $S$ can be any subset of $N$) can be solved simply by taking the union of optimal unconstrained assortments for each type of MNL. 
}

\noindent \textbf{Markov Choice Model }\citep{blanchet2016markov}: In this model, the customer choice process is described by a discrete markov chain on the state space of products (and the outside option). Given an assortment $S$, the customer starts at product $i\in N$ with probability $\lambda_i$. A customer at a product $i\not\in S$, transitions to product $j\in N$ with probability $\rho_{ij}$ (independent of their actions prior to $i$). The random process terminates when the customer reaches a product in $S\cup\{\emptyset\}$, which is their final choice. This model generalizes a wide array of choice models (see \citet{blanchet2016markov}). 

\noindent \emph{Previous results for assortment optimization:} \citet{blanchet2016markov} introduced this model and gave a polynomial time algorithm for unconstrained optimization. 
\citet{desir} consider the problem under cardinality and budget constraints. They show that the problem is APX hard under cardinality constraint and inapproximable under TUM constraints (sharp contrast with MNL). They obtain a $0.5-\epsilon$ approximation for the cardinality constrained problem and a  $1/3-\epsilon$ approximation for the budget constrained version. 

MNL, MMNL, and Markov choice models have received significant attention in the assortment optimization literature and even the simpler of these models i.e., MNL, is very useful for modeling choice behavior in practice (for instance, see \citet{feldman2018customer}). 
\subsection{New Results for Assortment Optimization}\label{sec:asstresult}


An obvious obstacle in the application of our results 
is the absence of monotonicity in the revenue objective. 
We solve this issue by redefining the objective as follows,
\[f_\phi(S) = \max_{X\subseteq S} R_\phi(X). \]
{\color{red}
%
}
We show that the function $f_{\phi}$ is monotone and subadditive under the following (mild) condition,
\[\emph{Substitutability: }\quad  \phi\left(i,S\cup\{j\}\right)\leq \phi(i,S),\, \forall i\in S, j\not\in S.  \]
MNL, MMNL, and Markov model, all satisfy this condition. In order to evaluate $f_{\phi}(S)$, we need to solve an unconstrained assortment optimization problem on the reduced ground set $S$. Recall that previous work gives polynomial time algorithms for unconstrained optimization under the choice models of interest to us. Therefore, we have efficient implementation of the value oracle for $f_{\phi}$. More generally, if a choice model $\phi$ is such that $f_{\phi}$ has a submodular order but the unconstrained assortment problem only admits a $1-\delta$ approximation (instead of an efficient optimal algorithm), then the approximation guarantees are reduced by a multiplicative factor of $\left(1-O(\frac{n\delta}{1-\delta})\right)$. This factor is negligible for small $\delta$ and this can be useful in case the unconstrained assortment problem admits a FPTAS. 

{\color{black}
\subsubsection{Results for MNL and MMNL with Customization.}
When $\phi$ is MNL, we show that sorting products in descending order of price (breaking ties arbitrarily) is a (strong) submodular order for $f_{\phi}$. 
Consequently, all our results for constrained optimization of submodular order apply directly. 
More importantly, using the fact that submodular order functions are closed under addition (see Remark \ref{closed}), 
we immediately obtain the first constant factor approximation for the problem of assortment optimization with customization under MMNL choice. 
\begin{remark} \label{closed}
Given $p$ monotone subadditive functions $f_1,\cdots,f_m$ with the same submodular order $\pi$ and non-negative real values $\alpha_1,\cdots,\alpha_m$, the function $\sum_{j\in[m]}\alpha_jf_j$ is also monotone subadditive with submodular order $\pi$. 
\end{remark}


\begin{theorem}\label{mnlresult}
The problem of assortment optimization with customization under MMNL choice is an instance of cardinality constrained submodular order maximization. 
\end{theorem}
As a direct consequence of Theorem \ref{mnlresult}, we have a $0.5-\epsilon$ approximation for any $m$. Table \ref{summary2} provides approximation guarantees under more general constraint on the selection of products (budget and matroid). 
We establish submodular order property for MNL model and prove Theorem \ref{mnlresult} in Appendix \ref{sec:asstmnl}. 
In fact, as a corollary of the result for matroid constraint, 
we obtain an approximation result for the problem of joint \emph{pricing} and customization under MMNL choice (see Appendix \ref{sec:asstmnlpricing} for more details). 

}

\subsubsection{Results for Markov Model and Beyond.}

The order given by descending prices is not a submodular order for every choice model. The following example demonstrates this for the Markov model.
\smallskip

\noindent \textbf{Example:} Consider a ground set of 4 items indexed $i\in[4]$ in decreasing order of prices $r_1=8,r_2=4,r_3=4$ and $r_4=2$. Customer chooses in a markovian fashion starting at item 2 with probability 1. If item 2 is not available, the customer transits to item $j$ with probability $1/3$ for every $j\in \{1,3,4\}$. If $j$ is also unavailable then the customer departs with probability 1. Consider the sets $B=\{1,2\}$ and $A=\{1,2,3\}$. We have $f(A)=f(B)=4$. However, if we add item 4 to these sets we get, $f(4\mid A)=R(\{1,3,4\})-4=2/3$, whereas, $f(4\mid B)=R(2)-R(2)=0$. 

It is not obvious (to us) if there is an alternative submodular order for Markov choice model. We give a procedure that extracts a \emph{partial} submodular order for any given Markov choice model and show that all our algorithms for submodular order functions apply to assortment optimization in this model without any loss in guarantee. In fact, we show this more generally for any choice model that has the following structure.
\smallskip

\noindent \textbf{Compatible choice models:} Given a substitutable choice model $\phi$, let $S$ be an optimal unconstrained assortment on the ground set $N$. 
We say that $\phi$ is \emph{compatible} if,
\begin{eqnarray}
R_{\phi}(A\mid C)&\geq &0 \qquad  \forall A\subseteq S,\, C\subseteq N \label{prop2}\\
R_{\phi}(C\mid A)&\leq &R_{\phi}(C\mid B) \qquad \forall B\subseteq A\subseteq S,\, C\subseteq N \label{prop1}
\end{eqnarray}
\emph{Compatibility} is a structural property of optimal unconstrained assortments that gives sufficient conditions (in the absence of a submodular order) for translating algorithms (with guarantees) from submodular order maximization. 
\begin{theorem}\label{compatible}
The family of Markov choice models 
is compatible. For any compatible choice model $\phi$ that admits a polynomial time algorithm for finding optimal unconstrained assortments, Algorithm \ref{framework} runs in polynomial time and matches the guarantee obtained by algorithms for submodular order functions under cardinality, budget, and matroid constraint (see table \ref{summary2}).  
\end{theorem}

Algorithm \ref{framework} is introduced in Section \ref{sec:framework} and it builds on the algorithms for submodular order functions. We prove Theorem \ref{compatible} in Appendix \ref{asst:markov}.
As an application of the above theorem, we obtain the first non-trivial approximation for joint pricing and assortment optimization in the Markov model (see Appendix \ref{asst:markov} for a proof).
\begin{corollary}\label{markovcoro}
There is a 0.25 approximation for joint pricing and assortment optimization in the Markov choice model. 
\end{corollary}

\begin{table}{}	 

\centering
\begin{tabular}{c|c|c|}
\multirow{2}{*}{}   & 
\multirow{2}{*}{MNL with customization}
&{Compatible choice models}\\       
&& \small{(includes Markov model)}\\                                                     
\hline
\multirow{2}{*}{	Cardinality} &\multirow{2}{*}{$\frac{\Omega(1)}{\log m}\to\mb{0.5-\epsilon}$} & 
\multirow{2}{*}{$ 0.5 -\epsilon$} \\
&&\\  
\cline{1-3}                             
\multirow{2}{*}{	Budget/Knapsack/Capacity    }        &\multirow{2}{*}{ $\frac{1}{m} \to \mb{0.5 -\epsilon}$}
&\multirow{2}{*}{ $\frac{1}{3}-\epsilon
\to\mb{0.5-\epsilon}$}\\ 
&&\\
\cline{1-3} 
Matroid      &\multirow{2}{*}{{$\frac{1}{m} \to \mb{0.25}$} } &\multirow{2}{*}{$\frac{1}{d} \to \mb{0.25}$}\\
\small{(Includes Pricing and Assortment)} &&\\
\cline{1-3}
\\ 	
\end{tabular}
\caption{Results for constrained assortment optimization. New results are stated in bold and previous state-of-the-art provided to the left of each arrow (including results implied from but not explicitly stated in previous work). Note that $m$ denotes the number of MNL models in the mixture and $d$ is the matroid rank.} 
\label{summary2}
\end{table}

\subsubsection*{Summary of Application to Assortment Optimization.} The framework of submodular order functions provides a new algorithmic tool for constrained assortment optimization. 
Given a choice model (that is not MNL or Markov), we 
summarize the key steps to check if 
this tool can be applied to obtain efficient approximations. 
\begin{enumerate}[(i)]
\item Is the unconstrained assortment problem efficiently solvable? If not, is there a FPTAS?
\item Is there a (strong or weak) submodular order? Specifically, is the descending order of revenues a submodular order?
\item If  the unconstrained assortment problem is efficiently solvable but a submodular order is not evident (or does not exist), is the choice model compatible?
\end{enumerate}
\noindent Next, we discuss a very different application of submodular order maximization. 
\subsection{Application to Streaming Submodular Maximization}\label{sec:introstream}

\noindent \textbf{The streaming model:} Consider a setting where the ground set $N$ is ordered in some arbitrary manner, say $\{1,2,\cdots,n\}$. We have a monotone submodular objective $f$ that we want to maximize subject to some constraint. Due to the large size of $N$, elements can only be accessed sequentially in order, i.e., to access element $j$ we need to parse the data stream of elements from $1$ to $j$. We also have a small working memory. Thus, we seek an algorithm that guarantees a good solution with very few passes (ideally, just one pass) over the data stream and requires very low memory (ideally, $\tilde{O}(k)$, where $k$ is the size of a maximal feasible solution). This setting has a variety of applications in processing and summarizing massive data sets \citep{cardstream}. Note that standard algorithms for submodular maximization, such as the greedy algorithm for cardinality constraint, make $\Theta(kn)$ passes over the data stream given small working memory.  
\smallskip

\noindent \emph{Previous work in streaming model:} 
The first result in this model was given by \cite{kalestream}, who gave a $\frac{1}{4p}$ approximation for maximization subject to intersection of $p$ matroids. \cite{cardstream} gave a different algorithm with improved guarantee of 
$0.5-\epsilon$ for cardinality constraint. \cite{feldman2020one} show that every streaming algorithm with guarantee better than $0.5+\epsilon$ requires $O(\epsilon n/k^3)$ memory.  \cite{budgetstream} gave a $0.4-\epsilon$ streaming algorithm for budget constrained optimization. Recently, \cite{feldman22} gave improved approximation algorithms for matroid constraint. 
\cite{chekuristream}  
showed results for very general ($p-$matchoid) constraints as well as non-monotone submodular functions. A more comprehensive review of related work can be found in \cite{budgetstream, feldman22}. 
\smallskip

{\color{black}	\noindent \textbf{Connection to submodular order maximization:} Consider a function $f$ that is $\pi$-submodular order. Given the salience of order $\pi$, many of our algorithms for optimizing $f$ are intentionally designed so that they parse elements in the order given by $\pi$. In fact, these algorithms can be efficiently implemented such that they parse the ground set exactly once and require very little memory. 
As a corollary, we recover the best known guarantees for streaming submodular maximization in the cardinality constraint case, as well as, constant factor approximation results for budget and matroid constraint.  
\begin{theorem}\label{stream}
For constrained maximization of a monotone submodular function $f$,
\begin{enumerate} [(i)]
\item Algorithm \ref{calg} gives a $(1-\epsilon)\, 0.5$ approximation in the streaming setting.
\item Algorithm \ref{bcalg} gives a $\frac{1}{3}-\epsilon$ approximation algorithm in the streaming setting.
\item Algorithm \ref{modmalg} gives a $0.25$ approximation in the streaming setting.
\end{enumerate}  
\end{theorem}
We include the proof in Appendix \ref{sec:streamproof}.  Note that Algorithm \ref{5balg}, which is $0.5-\epsilon$ approximate for budget constraint, requires working memory polynomial in $n$ and is not a streaming algorithm.  The main idea behind Theorem \ref{stream} is straightforward. Recall that a submodular function satisfies the submodular order property on every permutation of the ground set. Therefore, an $\alpha$ approximate algorithm for submodular order functions that parses the ground set only once (in submodular order) and requires low memory is, by default, an $\alpha$ approximation algorithm for streaming submodular maximization. While the connection is obvious in hindsight, it raises intriguing new questions. For example, the upper bound of 0.5 for submodular order functions (Theorem \ref{impossible}) is a consequence of their milder structure. In contrast, the upper bound of $0.5$ for streaming maximization (\cite{feldman2020one}) arises out of restrictions on memory. From a purely structural viewpoint, submodularity permits a stronger $(1-1/e)$ guarantee. \emph{Is there a precise connection between less structure and memory limitations?}} 

\section{Algorithms for Submodular Order Functions}\label{intuit}

In Example 1 we saw that in the absence of submodularity, iterative algorithms that augment the solution at each step (such as the family of greedy algorithms) may add elements that have large marginal value but \emph{blind} the algorithm from picking up more ``good" elements. Submodularity mutes this problem by guaranteeing that for any sets $S$ and $A$, if $f(A)>f(S)$ then there exists an element $e\in A\backslash S$ such that, $f(e|S)\geq\frac{f(A)-f(S)}{|A|}$. 
In general, the absence of submodularity is severely limiting (see Theorem \ref{mor}). However, for functions with submodular order we have a lifeline. Consequently, all our algorithms treat the submodular order as salient and ``process" elements in this order. 
Broadly speaking, we propose two types of algorithms. The first is inspired by greedy algorithms and the second by local search.

\subsection{Threshold Based Augmentation in Submodular Order}

Consider an instance of the cardinality constrained problem for a $\pi$ submodular function. Fix an optimal solution and given set $S$, let $\opt_S$ denote the subset of optimal products located to the right of $S$ in the submodular order ($l_{\pi}(\opt_S)>r_{\pi}(S)$). Suppose we have an algorithm that maintains a feasible set at every iteration and let $S_j$ denote this set at iteration $j$. Using weak submodular order property (and monotonicity), we can show that there exists $e$ to the right of $S_j$ such that
\[f(e\mid S_j)\geq \frac{f(S_j\cup \opt_{S_j})-f(S_j)}{k}.\]
Now, if we greedily add an element to $S_j$, we may add an element very far down in the order. This may substantially shrink the set $\opt_{S_{j+1}}$ in the next iteration. To address this issue we consider a \emph{paced} approach where we add the first element $e$ to the right of $S_j$ such that
\[f(e\mid S_j)\geq \tau(N,f,k,S_j),\]
where $\tau(N,f,k,S_j)$ is a \emph{threshold} that can depend on the instance $N,f,k$ as well as set $S_j$. In other words, we take marginal values into consideration by focusing only on elements that have sufficiently large value (but not the maximum value). From this filtered set of candidates we add the closest element to the right of $S_j$. Thus, minimizing the shrinkage in $\opt_{S_{j+1}}$. 

The simplest possible threshold is a constant independent of the iteration. Quite surprisingly, for cardinality constraint, choosing a constant threshold leads to the best possible guarantee for maximizing (strong and weak) submodular order functions. Letting \opt\ denote the optimal value, any threshold close to $\frac{\opt}{2k}$ is a good choice. Although we do not know \opt\ a priori, trying a few values on a geometric grid brings us arbitrarily close.  {\color{black}In Algorithm \ref{thalg}, let the ground set $N=\{1,2,\cdots,n\}$ be indexed in submodular order.}


\begin{minipage}{.45\linewidth}
\begin{algorithm}[H]
	\SetAlgoNoLine
	\textbf{Input:} Cardinality $k$, error $\epsilon\in(0,1)$\;
	\smallskip
	Initialize $\tau=\frac{1}{k}\max_{e\in N} f(\{e\})$\;
	\For{$i\in\{1,2,\cdots,\ceil{\log_{1+\epsilon} k}\}$}{
		$S_i=\text{Threshold Add}\left(k,(1+\epsilon)^{i-1}\tau\right)$\;
	}
	\smallskip
	\textbf{Output:} Best of $\{S_1,S_2,\cdots,S_{\ceil{\log_{1+\epsilon} k}}\}$
	\caption{$\frac{1}{2}$ for Cardinality}
	\label{calg}
\end{algorithm}
\end{minipage}\hspace{1cm}
\begin{minipage}{.45\linewidth}
\begin{algorithm}[H]
	\SetAlgoNoLine
	\textbf{Input:} Cardinality $k$, threshold $\tau$, {\color{black}$N$ indexed in submodular order}\;
	\smallskip
	Initialize $S=\emptyset$\;
	\For{$i\in\{1,2,\cdots,n\}$ and $|S|<k$}{
		\lIf {$f(i|S)\geq \tau$} {$S\to S \cup \{i\}$\;}	
	}
	\smallskip
	\textbf{Output:} Set $S$ of size at most $k$\;
	\caption{Threshold Add($k,\tau)$}
\label{thalg}
\end{algorithm}
\end{minipage}

\smallskip

\noindent \emph{Remarks:} The idea of iteratively adding elements with marginal value above a threshold has a rich history in both submodular optimization and assortment optimization. \cite{badanidiyuru2014fast} gave a fast $(1-1/e)-\epsilon$ approximation for constrained maximization of submodular functions using adaptive thresholds. \cite{cardstream} proposed an adaptive threshold algorithm for 
cardinality constraint in the streaming model. 
In the assortment optimization literature, \cite{desir} introduce a constant threshold algorithm for cardinality and budget constrained assortment optimization in the Markov model. 
\smallskip

\noindent \textbf{Generalizing to budget constraint:} Algorithm \ref{bcalg} presents a natural generalization with a threshold on the ``bang-per-buck", i.e., \emph{marginal value per unit budget}. Formally, we filter out elements $e$ such that $\frac{f(e\mid S_j)}{b_e}< \tau(N,f,B)$. We show a guarantee of $1/3$ for Algorithm \ref{bcalg}, short of the upper bound of 0.5 implied by Theorem \ref{impossible}. 

A standard technique in submodular maximization (and beyond) to obtain the best possible guarantees under budget constraint is \emph{partial enumeration}. For example, Sviridenko \cite{sviridenko2004note} considers all possible sets of size 3 as a starting point for a greedy algorithm that picks an item with best bang-per-buck at each step thereafter. Starting with an initial set $X$ is equivalent to changing the objective function to $f(\cdot\mid X)$. In case of submodular order functions $f(\cdot\mid X)$ will, in general, not have a weak submodular order even for a small set $X$. 
To preserve submodular order, we may restrict enumeration to starting sets that are concentrated early in the order. However, this limited enumeration appears to be ineffective. Instead, we propose an algorithm that starts with an empty set and parses elements in submodular order but enumerates over a small set of high budget elements which get ``special attention" via the Final Add subroutine. It should be noted that unlike Algorithm \ref{bcalg}, which only adds elements, Algorithm \ref{5balg} may discard elements that were added in previous iterations via the Final Add subroutine. 

\begin{minipage}{0.45\linewidth}
\begin{algorithm}[H]
\SetAlgoNoLine
\textbf{Input:} Budget $B$, error $\epsilon\in(0,1)$\;
\smallskip
Initialize $\tau=\frac{1}{B}\max_{e\in N} f(\{e\})$\;
\For{$i\in\{1,2,\cdots,\ceil{\log_{1+\epsilon} n}\}$}{
$S_i=\text{B-Th.\ Add}\left(B,(1+\epsilon)^{i-1}\tau\right)$\;
}
\smallskip
\textbf{Output:} Best of all \emph{singletons} and sets $\{S_1,S_2,\cdots,S_{\ceil{\log_{1+\epsilon} k}}\}$
\caption{$\frac{1}{3}$ for Budget }
\label{bcalg}
\end{algorithm}
\end{minipage}\hspace{1cm}
\begin{minipage}{0.45\linewidth}
\begin{algorithm}[H]
\SetAlgoNoLine
\textbf{Input:} Budget $B$, threshold $\tau$, $N$ indexed in submodular order\;
\smallskip
Initialize $S=\emptyset$\;
\For{$i\in\{1,2,\cdots,n\}$ and $b(S)<B$}{
\lIf {$b_i<B-b(S)$ and $\frac{f(i|S)}{b_i}\geq \tau$} {$S\to S \cup \{i\}$}	
}
\smallskip
\textbf{Output:} Feasible set $S$\;
\caption{B-Th.\ Add($B,\tau)$}
\label{bth}
\end{algorithm}
\end{minipage}
\begin{minipage}{0.6\linewidth}
\begin{algorithm}[H]
\SetAlgoLined
\textbf{Input:} Budget $B$, error $\epsilon\in(0,1)$, $N$ indexed in submodular order\;
\smallskip
Initialize collection $E=\left\{X\subseteq N\,\big|\, b(X)\leq B,\, |X|\leq 1/\epsilon\right\}$\; 
\For{$X\in E$}{
Filter ground set $N\to N\backslash \{e\mid b_e> \min_{i\in X} b_i\}$\;
Initialize $\tau=\frac{1}{B}\max_{e\in N} f(\{e\})$\;
\For{$i\in\{1,2,\cdots,\ceil{\log_{1+\epsilon} |N|}\}$}{
Initialize $S_i=\emptyset$\;
\For{$j\in\{1,2,\cdots,|N|\}$ and $b(S_i)<B$}{
		\If{$\frac{f(j|S_i)}{b_i}\geq (1+\epsilon)^{i-1}\tau$}{
			\eIf {$b_j+b(S_i)\leq B$} {$S_i\to S_i \cup \{j\}$}
			{
					Update $S_i\to \textbf{Final Add } (B,\epsilon,S_i\cup\{j\})$\; 
					Goto next $i$\;	
			}}
}}

Let $S(X)=$ Best of all feasible sets $S_i$\;
}
\smallskip
\textbf{Output:} Best of all sets $\{S(X) \mid X\in E\}$
\caption{$0.5$ for Budget}
\label{5balg}
\end{algorithm}
\end{minipage}
\begin{minipage}{0.4\linewidth}
\begin{algorithm}[H]
\SetAlgoLined
\textbf{Input:} Budget $B$, $\epsilon$, set $S$\;
\smallskip
\While{$b(S)>B$}{
Remove an element $i$ with budget $b_i< \epsilon B$ from $S$\; 
End loop if no such element exists\;
}
\textbf{Output:} $S$\;
\smallskip
\caption{Final Add}
\label{fbalg}
\end{algorithm}
\end{minipage}
\smallskip

\noindent The following example illustrates why threshold based algorithms fail under matroid constraints.
\smallskip

\noindent \textbf{Example:} Consider a matroid with rank $2^n-1$ on the ground set $\{1,\cdots,2^{n+1}-2\}$, indexed in submodular order. Let $R$ denote the set $\{1,\cdots,2^{n}-1\}$, that is, the first half of the ground set. Let $P$ denote the other half. Partition $R$ into ordered sets $\{R_1,\cdots,R_n\}$ such that $R_1$ is the set $[2^{n-1}]$ of the first $2^{n-1}$ elements, $R_2$ of the next $2^{n-2}$ elements, etc.\ Therefore, $|R_i|=2^{n-i}$ for every $i\in[n]$. Partition $P$ into ordered sets $\{P_1,\cdots,P_n\}$ so that $P_i$ is the set $\{2^{n}-1+2^{i-1},2^{n}-2+2^{i}\}$ of $2^{i-1}$ elements. Observe that $|R_i|=|P_{n-i+1}|$. We define independent sets in the matroid so that the 
set $R_i\cup \left(\cup_{j=n-i+1}^{n} P_j\right)$ is independent for every $i\in[n]$. Therefore, $P$ is independent and so is the set $R_i$ (and $P_i$) for every $i$. Additionally, let the rank of the set $R_i\cup R_{i+j}$ equal $|R_i|$ for every $j\geq 1$.  
Finally, define a \emph{modular} function $f$ so that $P$ is optimal and $f(P)=2^{n}-1$. Let $f(P_i)=f(P)/n$ and $f(e)=f(P_i)/|P_i|$, for every $i\in[n]$ and $e\in P_i$. Similarly, let $f(R_i)=f(P_{n-i+1})=f(P)/n$ and $f(e)=f(R_i)/|R_i|$, for every $i\in[n]$ and $e\in R_i$. Now, for any given threshold $\tau$, the algorithm that parses the elements in order and selects every element with marginal value exceeding $\tau$ (while maintaining independence) will pick a set with value exactly $2f(P)/n$ in this instance (translating to a factor $2/\log n$ for ground set of size $n$).


\subsection{Local Search Along Submodular Order}

To address the challenges with matroid constraint, we introduce an algorithm inspired by local search. 
Given a feasible set $S$, a local search algorithm tries to make ``small" changes to the set in order to improve function value. Often, this is done via \emph{swap operations} where a set $X$ of elements is added to $S$ and a subset $Y$ (can be empty) is removed such that $f(S\cup X\backslash Y)\geq (1+\epsilon)f(S)$, for some $\epsilon>0$, and the new set $S\cup X\backslash Y$ is feasible\footnote{
Other local search techniques exist. For instance, oblivious local search algorithms (see \cite{filmus2012tight}) may perform swaps to increase the value of a proxy function instead of the objective.}. The algorithm terminates at local maxima and for submodular functions the local maxima are within a constant factor of optimal \cite{feige2011maximizing}. 
As illustrated in Example 1, functions with submodular order have poor local maxima. 
Therefore, we perform a \emph{directed} local search where we consider elements for swap operations one-by-one in the submodular order. 
Consider the following algorithm based on this idea, 
\smallskip

\noindent \emph{Parse elements in submodular order and add $j$ to current set $S$ if there exists an element $i$ (could be $\emptyset$) such that $S':=(S\backslash\{i\})\cup\{j\}\in \mathcal{I}$ and $f(S')\geq (1+\epsilon)f(S)$.}

At first, one may prefer to set a small value of $\epsilon$. This would entail slower convergence but is expected to yield stronger approximation guarantee. However, unlike standard local search, the algorithm above parses each element exactly once. It turns out that unless $\epsilon$ is large enough, a single pass over the elements may end with a poor chain of swaps $i_1\to i_2\cdots\to i_k$. This is demonstrated in the next example. 

\noindent \textbf{Example:} Consider the independent set $S=\{1,\cdots,m\}$ with $f(S)=1$ at iteration $m+1$. Let $\{m,m+1\}$ form a circuit and suppose $m+1$ meets the criteria for a swap when $1/\epsilon=O(m^2)$ (quite small). Therefore, we have $S=[m-1]\cup\{m+1\}$ after iteration $m+1$ and let $f(S)=1+\epsilon$. In subsequent iteration, suppose $m+2$ makes a circuit with element $2\in S$ and let $f(S\cup\{m+2\}\backslash\{2\})=f(S)$. Thus, $m+2$ does not meet the criteria for a swap. However, $f(\{m,m+2\})=2$. In the next iteration, we swap out $m+1$ in favor of $m+3$ and for the resulting set let $f(S)=1+2\epsilon$. Next, we find that element $m+4$ does not meet the criteria for a swap but $f(\{m,m+2,m+4\})=3$. Inductively, after $2m$ such iterations we have an instance such that the local search solution value is $1+m\epsilon\approx 1+1/m$, while the optimal value is $m$.

At a high level, we prevent such poor chains of swaps from occurring by constantly adjusting the swap criteria so that the $k$-th element in a chain $i_1\to\cdots\to i_k$ of swaps has marginal value greater than equal to the 
\emph{sum of the marginal values of all previous elements} in the chain. 

\begin{algorithm}[H]
\SetAlgoLined
\textbf{Input:} Independence system $\mathcal{I}$, $N$ indexed in submodular order\; 
\smallskip
Initialize $S,R=\emptyset$ and values $v_j=0$ for all $j\in[n]$\;
\For{$j\in\{1,2,\cdots,n\}$}{
\lIf {$S\cup\{j\}\in \mathcal{I}$} {initialize $v_j=f(j\mid S\cup R)$ and update $S\to S \cup \{j\}$}
\Else{ 
Find circuit $C$ in $S\cup\{j\}$ and define $i^*=\underset{i\in C\backslash\{j\}}{\arg\min}\, v_i$ 
and $v_C=v_{i^*}$\; 
\If  {$f(j|S\cup R)
	>v_C$} { 
	$v_j\to v_C + f(j\mid S\cup R)$\;
	$S\to (S\backslash\{i^*\})\cup \{j\}$ and $R\to R \cup \{i^*\}$\;
}}}
\smallskip
\textbf{Output:} Independent set $S$
\caption{$\frac{1}{4}$ for Matroid} 
\label{malg}
\end{algorithm}
\smallskip

\noindent \emph{Remark:} For functions with strong submodular order, we can use marginals $f(j\mid S)$ instead of $f(j\mid S\cup R)$ without changing the approximation guarantee. The modified algorithm (Algorithm \ref{modmalg}) and its performance analysis is included in Appendix \ref{sec:streamproof}. The fact that set $S\cup R$ grows monotonically in Algorithm \ref{malg} is critical in application to assortment optimization of compatible choice models. On the other hand, using $f(j\mid S)$ makes it possible to implement Algorithm \ref{modmalg} in the streaming setting where memory is small ($\tilde{O}(d)$).
\subsection{Algorithms for Constrained Assortment Optimization}\label{sec:framework}
In this section, we translate the algorithms from submodular order maximization to constrained assortment optimization. Recall  
that the revenue objective in assortment optimization can be a non-monotone function. So we transform the objective to $f_{\phi}(S)=\max_{X\subseteq S} R_{\phi}(S)$, which is monotone and subadditive for choice models that are substitutable. In an ideal scenario, we have a choice model $\phi$ where the unconstrained assortment optimization can be solved efficiently (or admits a FPTAS) and $f_{\phi}$ has a submodular order.  
If a submodular order is not evident, or does not exist, but the choice model is compatible (a property of optimal unconstrained assortments), we propose a new framework. 

Our framework for 
compatible choice models builds on the algorithms for submodular order functions and sequentially constructs an order over the ground set, satisfying a relaxed notion of the submodular order property. 
To define the framework, we need some notation. Let $\mathcal{A}$ represent an algorithm for constrained submodular order maximization (out of Algorithms \ref{calg}, \ref{bcalg}, \ref{5balg} and \ref{malg}). An instance of the assortment optimization problem is given by set $N$, function $f_{\phi}$, and constraint $\mathcal{F}$.
Let $\Gamma_{\mathcal{A}}(N,f_{\phi},\mathcal{F})$ define the set of parameter settings that $\mathcal{A}$ enumerates. We explicitly define this set for each algorithm later. 
For brevity, we write $\Gamma_{\mathcal{A}}(N,f_{\phi},\mathcal{F})$ simply as $\Gamma_{\mathcal{A}}$. 
Given a parameter setting $\gamma\in \Gamma_{\mathcal{A}}$, 
let 
\[(S,R):=\mathcal{A}_{\gamma}(f,N,\pi,\mathcal{F}),\] 
denote the output of $\mathcal{A}$ 
when ground set $N$ is ordered according to $\pi$. 
The set $S$ is the feasible solution generated by $\mathcal{A}_{\gamma}$. $R$ is a minimal set such that every query made by $\mathcal{A}_{\gamma}$ during its execution is a marginal value $f(\cdot\mid X)$ for some $X\subseteq S\cup R$. Finally, let $U_{\phi}(X)$ denote an optimal unconstrained assortment on ground set $X$. Observe that $f_{\phi}(X)$ is the expected revenue of $U_{\phi}(X)$. Before stating our new framework, we explicitly define $\Gamma_{\mathcal{A}}$ and $R$ for each algorithm.
\smallskip

\noindent \textbf{$\Gamma_{\mathcal{A}}$ and $R$ for Algorithms \ref{calg}, \ref{bcalg}, and \ref{5balg}:} $\Gamma_\mathcal{A}$ is the set of all threshold values $\tau$ tried in the algorithm. For Algorithm \ref{5balg}, the set also includes sets $X$ of high budget elements  used to filter the ground set. In Algorithms \ref{calg} and \ref{bcalg} marginal values are evaluated only w.r.t.\ subsets of the output $S$. So $R=\emptyset$ for every parameter setting. In Algorithm \ref{5balg}, $R$ is the set of elements discarded by the Final Add subroutine (Algorithm \ref{fbalg}).
\smallskip

\noindent \textbf{$\Gamma_{\mathcal{A}}$ and $R$ for Algorithm \ref{malg}:} There is no enumeration in Algorithm \ref{malg} so $\Gamma_{\mathcal{A}}$ is a singleton set (the \textbf{for} loop executes once). $R$ is the set of all elements that were picked by the algorithm and then swapped out. 
\medskip


\begin{algorithm}[H]
\SetAlgoLined
\textbf{Input:} Algorithm $\mathcal{A}$, set $N$, constraint $\mathcal{F}$, oracles for $f_\phi,U_\phi$\; 
\smallskip
\For{$\gamma \in \Gamma_{\mathcal{A}}$}{
Initialize $\hat{N}=U_{\phi}(N)$, $M=\emptyset$ and arbitrary order $\pi_{\hat{N}}:\hat{N} \to \{1,\cdots,\hat{N}\}$\;
\While{$\hat{N}\backslash M \neq \emptyset$}{
$(S,R)=\mathcal{A}_{\gamma}\left(f_{\phi},\hat{N},\pi_{\hat{N}},\mathcal{F}\right)$\;
$M= S\cup R$\;
$N\to (N\backslash \hat{N})\cup M$\;
$\hat{N}\to U_{\phi}(N)\cup M$\; 
Update $\pi_{\hat{N}}$ by adding elements in $\hat{N}\backslash M$ last in the order\; 
}
$S_\gamma:=S$
}
\smallskip
\textbf{Output:} Best of $\{S_{\gamma}\}_{\gamma\in \Gamma_{\mathcal{A}}}$
\caption{Constrained Assortment Optimization for Compatible Choice Models} 
\label{framework}
\end{algorithm}
\medskip



At a high level, given a constrained assortment optimization problem without a submodular order, the framework sequentially constructs a (partial) submodular order. It starts with the unconstrained optimal assortment $U_{\phi}(N)$, denoted as $N_1$ for brevity. 
For a compatible choice model, we show that 
every order $\pi$ that places elements in $N_1$ before all other elements, defines a partial submodular order in the following sense, 
\[f(C\mid A)\leq f(C\mid B)\qquad \forall\, B\subseteq A\subseteq N_1,\,  C\subseteq N.\]
Given this insight and the fact that our algorithms for submodular order functions parse the ground set in submodular order, the framework now executes $\mathcal{A}$ with the truncated ground set $N_1$. 
Let $M_1=S\cup R$ denote the output of $\mathcal{A}$ after this first phase. 
The set $N_1\backslash M_1$ of elements that are parsed and rejected by $\mathcal{A}$ are now discarded and we update $N\to (N\backslash N_1)\cup M_1$, which is the set of remaining elements. Having parsed elements in $N_1$, the framework now augments the (partial) submodular order by re-solving for an unconstrained optimal assortment on the new ground set $N$. 
The new set $U_{\phi}(N)$ may now include elements that were not previously parsed by $\mathcal{A}$. Let $N_2=U_{\phi}(N)\backslash M_1$ denote this new set of elements.  The framework adds $N_2$ to the end of the partial submodular order (to the right of $N_1$) and executes $\mathcal{A}$ now on the ground set $M\cup N_2$. Repeating this process several times, new elements are incrementally passed into $\mathcal{A}$ in a growing partial order $\pi$. The algorithm terminates when no new elements appear in the unconstrained assortment (after at most $n$ phases). 

To summarize, Algorithm \ref{framework} employs the algorithms for submodular order maximization to solve a constrained assortment optimization when it is not evident that the choice model has a submodular order. Note that this framework can be used to solve constrained assortment optimization for any choice model which admits an efficient algorithm for solving the unconstrained problem. However, 
the approximation guarantees may not hold for non-compatible choice models. The analysis of Algorithm \ref{framework} is presented in Appendix \ref{asst:markov} (and Appendix \ref{appx:missing}). 

\section{Approximation Guarantees}\label{sec:proofsubmod} 
In this section, we analyze the performance of our algorithms for maximizing functions with (known) submodular order. In Section \ref{sec:upb}, we prove an upper bound (Theorem \ref{impossible}) on the approximation guarantee of efficient algorithms. For submodular functions, given a set $A$ and partition $\{O,E\}$ of $A$ we have, 
\[f(A)= f(E)+f(O|E)\,\leq\, f(E)+\sum_{i\in O} f(i\mid E).\] 
The inequality above is at the heart of proving tight guarantees for maximizing monotone submodular functions. 
For functions with submodular order this property need not hold unless all elements of $O$ are located entirely to the right of $E$ in the submodular order. In the following we show stronger upper bounds that play a crucial role in the analysis of all our algorithms. First, we define the notion of interleaved partitions. 
\smallskip 

\noindent \textbf{Interleaved partitions:} Given a set $A$ and an order $\pi$ over elements, an interleaved partition 
\[\{O_1,E_1,O_2,\cdots,O_{m},E_{m}\}\] 
of $A$ is given by sets $\{O_{\ell}\}_{\ell\in[m]}$ and $\{E_{\ell}\}_{\ell\in[m]}$ that alternate and never cross each other in the order $\pi$. Formally, for every $\ell\geq1$ we have $r_{\pi}(O_{\ell})<l_{\pi}(E_{\ell})\leq r_{\pi}(E_{\ell})<r_{\pi}(O_{\ell+1})$. To assist the reader we note that letters $O$ and $E$ signify odd and even numbered sets in the partition. As noted previously, for functions with submodular order the upper bound $f(A)\leq f\left(\cup_{\ell\in[m]}E_\ell \right)+f\left(\cup_{\ell\in[m]}O_\ell \mid \cup_{\ell\in[m]}E_\ell \right)$ does not always hold. Instead, we establish a family of (incomparable) upper bounds parameterized by permutation $\sigma:[m]\to[m]$. The flexibility provided by permutation $\sigma$ is particularly helpful in analyzing Algorithm \ref{5balg} and Algorithm \ref{malg}, where elements may be swapped out. 

The following set unions will be used extensively in the lemmas that follow. Given permutation $\sigma$, let
\[E(j)=\underset{\ell\leq j}{\bigcup}\, E_{\ell} \text{  and  } O_{\sigma}(j)=\underset{\ell\, \mid\, \sigma(\ell)\geq\, \sigma(j)}{\bigcup} O_{\ell},\] 
with $E(0):=\emptyset$ 
and $O_{\sigma}(m+1):=\emptyset$. When $\sigma(\ell)=\ell,\, \forall \ell\in[m]$, we use the shorthand $O(j)=\cup_{\ell\geq j}O_\ell$. In this notation, $E(m)=\cup_{\ell\in[m]}E_\ell$, $O(1)=\cup_{\ell\in[m]}O_\ell$. Thus, $A=O(1)\cup E(m)$. Next, for every $j\in[m]$ define 
\[L_\sigma(j)=\underset{\ell\, \mid\, \ell< j,\, O_\ell\subseteq O_\sigma(j)}{\bigcup} O_\ell,\] 
which is the union over sets $O_\ell$ that are to the left of $O_j$ in submodular order and to the right of $O_j$ in permutation $\sigma$. 
Notice that when $\sigma(\ell)=\ell,\, \forall \ell\in[m]$, we have $L_{j,\sigma}=\emptyset$ for every $j\in[m]$. 

\begin{lemma}\label{mainleaf}
Given a set function $f$ with weak submodular order $\pi$, set $A$ with interleaved partition $\{O_{\ell},E_{\ell}\}_{\ell=1}^m$, and a permutation $\sigma:[m]\to[m]$, 
we have
\[f(A) \leq f\left(E(m) \right) + \sum_{\ell\in[m]} f\left(O_{\ell}\mid  L_{\sigma}(\ell)\cup E(\ell-1)\right). \]
\end{lemma}
\begin{proof}{Proof.}
The following inequalities are crucial for the proof,
\begin{equation}\label{induce}
f\left( O_\sigma(\ell)\cup E(m)\right)\leq f\left(O_{\ell} \mid L_\sigma(\ell) \cup E(\ell-1) \right)+ f\left( O_\sigma(\ell) \cup E(m)\backslash O_\ell\right) \quad \forall\, \ell\in[m].
\end{equation}
Before we establish \eqref{induce},	observe that by summing these inequalities for $\ell\in [m]$ we get
\begin{eqnarray*}
\sum_{\ell\in[m]}f\left( O_\sigma(\ell)\cup E(m)\right)-\sum_{\ell\in[m]}f\left( O_\sigma(\ell) \cup E(m)\backslash O_\ell\right)&&\leq \sum_{\ell\in[m]}f\left(O_{\ell} \mid L_{\sigma}(\ell) \cup E(\ell-1) \right),\\
f\left( O(1)\cup E(m)\right)-f\left( E(m)\right) &&\leq \sum_{\ell\in[m]}f\left(O_{\ell} \mid L_\sigma(\ell) \cup E(\ell-1) \right),\\
f(A)-f\left( E(m)\right) &&\leq \sum_{\ell\in[m]}f\left(O_{\ell} \mid L_\sigma(\ell) \cup E(\ell-1) \right),
\end{eqnarray*} 
as desired. 
It remains to show \eqref{induce}. Consider arbitrary $\ell\in[m]$ and notice that sets $B:= L_\sigma(\ell) \cup E(\ell-1)$ and $A:= L_\sigma(\ell) \cup E(\ell-1)\cup O_{\ell}$ are $\pi$-nested. Moreover, the set $C:=O_\sigma(\ell)\cup E(m) \backslash A$ lies entirely to the right of $A$. From weak submodular order property we have,
$f\left(C \mid A\right) \leq f(C\mid B)$.
Consequently,
\begin{eqnarray*}
f\left( O_\sigma (\ell)\cup E(m)\right)&= &f\left(A \right)+ f\left(C \mid A \right)\\
&\leq &f\left(B \right)+ f\left(O_{\ell} \mid B\right) +f\left(C \mid B\right)\\
&= &f\left(O_{\ell} \mid B\right) + f\left(B\cup C\right)\\
&= &f\left(O_{\ell} \mid L_\sigma(\ell) \cup E(\ell-1) \right) + f\left(O_\sigma (\ell)\cup E(m)\backslash O_\ell \right).
\end{eqnarray*}
\hfill\Halmos \end{proof}
\noindent \emph{Remark:} The following corollaries of the general bound in Lemma \ref{mainleaf} suffice for all our analyses. In the first corollary we choose $\sigma$ to be identity. In the second, we let $\sigma$ be the reverse order permutation. 
\begin{corollary}\label{interleaf}
Given a set function $f$ with weak submodular order $\pi$, set $A$ with interleaved partition $\{O_{\ell},E_{\ell}\}_{\ell=1}^m$, and the permutation $\sigma(\ell)=\ell$ for every $\ell\in[m]$,  
we have
\[f(A) \leq f\left(E(m) \right) + \sum_{\ell\in[m]} f\left(O_{\ell}\mid  E(\ell-1)\right). \]
\end{corollary}
\begin{corollary}\label{forE}
Consider a set function $f$ with weak submodular order $\pi$ and set $A$ with interleaved partition $\{O_{\ell},E_{\ell}\}_{\ell=1}^m$. 
Let $L_\ell=\{e\in A\mid \pi(e)<l_{\pi}(E_\ell)\}$. Then, 
\begin{enumerate}[(i)]
\item $f(A)\leq f\left(E(m)\right) + \sum_{\ell\in[m]} f\left(O_\ell \mid E(\ell-1)\cup O(1)\backslash O(\ell) \right),$
\item $\sum_{\ell\in[m]} f\left(E_\ell \mid L_{\ell}  \right) \leq f\left(E(m) \right).$ 
\end{enumerate}
\end{corollary}
\begin{proof}{Proof.}
Applying Lemma \ref{mainleaf} with $\sigma[\ell]=m-\ell+1$ for $\ell\in[m]$, we have
\begin{eqnarray*}
f(A)&\leq &f\left(E(m)\right) + \sum_{\ell\in[m]} f\left(O_\ell \mid E(\ell-1)\cup O(1)\backslash O(\ell) \right).
\end{eqnarray*}
To complete the proof we use the decomposition,
\[f(A)=\sum_\ell \left[f\left(E_\ell\mid L_\ell\right) + f\left(O_\ell \mid E(\ell-1)\cup O(1)\backslash O(\ell) \right) \right].\]
\hfill\Halmos \end{proof}

\subsection{Cardinality Constraint}\label{sec:card}
\begin{proof}{Proof of Theorem \ref{rescard}.}
We start by calculating the number of queries made by Algorithm \ref{calg}.
The Threshold Add subroutine makes at most $n$ queries. Algorithm \ref{calg} calls the Threshold Add subroutine $\ceil{\log_{1+\epsilon}k}$ times. This results in $O(n\log_{1+\epsilon}k)=O(\frac{n}{\epsilon}\log k)$ queries.

Next, we establish the approximation guarantee. Let $\tau_i=\tau(1+\epsilon)^{i-1}$. We use \opt\ to denote both the optimal solution and function value. Notice that for $\opt\leq 2\max_{e\in N} f(\{e\})$, we have $f(S_{\ceil{\log_{1+\epsilon} k}})\geq (1-\epsilon)\,0.5\, \opt$. So from here on we let $\opt> 2\max_{e\in N} f(\{e\})$. Consequently, there exists an $i$ such that
\[(1-\epsilon)\frac{\opt}{2k} \leq \tau_i \leq \frac{\opt}{2k}.\]
We show that $f(S_i)\geq (1-\epsilon)\, 0.5\, \opt$. For convenience, we omit $i$ from the subscript and write $S_i$ simply as $S$ and $\tau_i$ as $\tau$. 

Let $k'$ denote the cardinality of set $S$. By definition of Threshold Add, \[f(S)\geq k'\tau.\]
When $k'=k$ this gives us $f(S)\geq (1-\epsilon)\, 0.5\, \opt$. So let $k'<k$.  
From monotonicity of the function we have, $\opt\leq f(\opt \cup S)$. So consider $\opt\cup S$ and its interleaved partition \[\{O_1,\{s_1\},O_2,\{s_2\},\cdots,\{s_{k'}\},O_{k'+1}\},\] where $s_j$ denotes the $j$th element added to $S$. Set $O_{j+1}$ contains all elements in \opt\ between $s_j$ and $s_{j+1}$, i.e., $\pi(s_j)<l_{\pi}(O_{j+1})$ and  $r_{\pi}(O_j)<\pi(s_j)$ for every $j\in [k']$. Note that some sets may be empty. Applying Corollary \ref{interleaf} on $\opt \cup S$ with $E_{j}=s_j$ for $j\in[k']$ we have,
\begin{eqnarray}
\opt&\leq &f\left(E(k')\right) + \sum_{\ell\in[k']} f\left (O_{\ell}\mid E(\ell-1) \right), \nonumber\\
&= &f(S)+\sum_{\ell\in[k']} f\left (O_{\ell}\mid E(\ell-1) \right).\label{card} 
\end{eqnarray}
Using weak submodular order we have for every $\ell\geq 1$,
\begin{eqnarray*}
f\left (O_{\ell}\mid E(\ell-1) \right)&\leq &\sum_{e\in O_{\ell}}f\left (e\mid E(\ell-1)\right) \leq \tau |O_{\ell}|, 
\end{eqnarray*}
{\color{black}	where the second inequality follows the fact that for $k'<k$, elements in $\opt\backslash S$ did not meet the criteria in Threshold Add (Algorithm \ref{thalg})}. Plugging this into \eqref{card} and using the upper bound $\tau\leq 0.5\opt/k$ we get,
\begin{eqnarray*}
\opt 
\leq f(S) + k\tau \leq f(S)+0.5\, \opt.
\end{eqnarray*}
This completes the proof. 
Now, suppose we have a noisy function oracle $\hat{f}$ such that, $(1-\delta)f(S) \leq \hat{f}(S)\leq (1+\delta) f(S)\,\, \forall S\subseteq N$. Observe that $f(e\mid S) \leq \frac{1}{1-\delta}\hat{f}(S+e)-\frac{1}{1+\delta}\hat{f}(S)\leq \hat{f}(e\mid S) + \frac{2\delta}{1-\delta} f(S+e)$. For $k'=k$, we have, $f(S)\geq \frac{1}{1+\delta}\hat{f}(S)\geq \frac{1}{1+\delta} k'\tau\geq \frac{1}{1+\delta}(1-\epsilon)0.5\,\opt.$ For $k'<k$, we have,
\[f\left (e\mid E(\ell-1) \right)\,\leq\, \hat{f}\left (e\mid E(\ell-1) \right)\, +\,2\frac{\delta}{1-\delta} f(e \cup E(\ell-1))\,<\, \tau+2\frac{\delta}{1-\delta}\opt\quad \forall e\in O_\ell.\] 
Therefore, $f(S)\geq \left(0.5-\frac{2k\delta}{1-\delta}\right) \opt$. Overall, in the presence of a noisy oracle we have have approximation guarantee  $\min\left\{\frac{0.5}{1+\delta}(1-\epsilon),\, 0.5-\frac{2k\delta}{1-\delta} \right\}=(1-\epsilon)\left(1-O(\frac{n\delta}{1-\delta})\right)\, 0.5$.

\hfill\Halmos \end{proof}


\subsection{Budget Constraint}\label{sec:budget}
\begin{proof}{Proof of Theorem \ref{resbudget} (i).} 

The query complexity of Algorithm \ref{bcalg} is identical to Algorithm \ref{calg}. To establish the approximation guarantee, observe that we ignore elements $i\in N$ with $b_i>B$. Also, let $\sum_{i\in N}b_i>B$ (otherwise picking the ground set is optimal). 
Let $\tau_i=\tau(1+\epsilon)^{i-1}$. Notice that for $\opt\leq 2\max_{e\in N} f(\{e\})$ the final solution has value at least $0.5\, \opt$ due to the singleton $\argmax_{e\in N} f(\{e\})$. So from here on we let $\opt> 2\max_{e\in N} f(\{e\})$. Consequently, there exists an $i$ such that
\[(1-\epsilon)\frac{2\opt}{3B} \leq\, \tau_i \leq\, \frac{2\opt}{3B}.\]
We will show that $f(S_i)\geq \frac{1-\epsilon}{3}\opt$. Let $B_i$ denote the total budget utilized by set $S_i$. By definition of B-Threshold Add, \[f(S_i)\geq B_i\tau_i.\]
Thus, for $B_i\geq \frac{1}{2}B$ we have $f(S_i)\geq \frac{1-\epsilon}{3} \opt$. So let $B_i<\frac{1}{2}B$. Similar to the proof of Theorem \ref{rescard}, consider the set $\opt\cup S_i$ and its interleaved partition \[\{O_1,\{s_1\},O_2,\{s_2\},\cdots,\{s_{\kappa-1}\},O_\kappa\cup\{i^*\},\cdots,\{s_{k_i}\},O_{k_i+1}\},\]
where $s_j$ is the $j$-th element added to $S_i$. The main difference is the possible existence of element $i^*$ (in \opt)	 that meets the threshold requirement but cannot be added to $S_i$ due to the budget constraint. Note that there can only be one such element as $B_i<\frac{1}{2}B$. Applying Corollary \ref{interleaf}, we have 
\[f(\opt\cup S_i) \leq f(S_i)+f(\{i^*\}\mid E(\kappa-1)) + \sum_{\ell\in[k_i+1]} f(O_\ell \mid E(\ell-1)). \]
We consider two cases based on existence of $i^*$. Let \alg\ denote the value of solution generated by the algorithm.

\noindent \textbf{Case I:} $i^*$ does not exist. Let $B_j$ denote the total budget of set $O_j$. Using the upper bound $f(O_j \mid E(j-1))\leq \tau_iB_j$ for every $j\in[k_i+1]$, we have
\[\opt\leq f(S_i) + \tau_i \sum_{j\in[k_i+1]} B_j = f(S_i) + \frac{2}{3}\opt. \]
Thus, $\alg\geq f(S_i)\geq \frac{1}{3}\opt$.

\noindent \textbf{Case II:} $i^*$ exists. Then, $f(\{i^*\}\mid E(\kappa-1))\geq \tau_i b_{i^*}$ and $b(S_i)+b_{i^*}>B$. Therefore,
\begin{equation}\label{max}
f(S_i)+f(\{i^*\}\mid E(\kappa-1))\geq \tau_i B \geq (1-\epsilon)\frac{2}{3} \opt.
\end{equation}
Now, the algorithm outputs the best of $S_i$ and all singletons. Therefore,
\[2\alg\geq  f(S_i)+f(i^*) \geq f(S_i)+f(\{i^*\}\mid E(\kappa-1)),\]
here second inequality follows from submodular order property. 
Using \eqref{max} completes the proof. 

Given a noisy function oracle $f$ with multiplicative error $(1\pm\delta)$, we have, $\hat{f}(S_i)\geq B_i\tau_i$. For $B_i\geq 0.5 B$, we get, $f(S_i)\geq \frac{1-\epsilon}{3(1+\delta)}\,\opt$. For $B_i<0.5 B$ and when $i^*$ does not exist, we have, $f(O_j \mid E(j-1))\leq \tau_iB_j+2\frac{|O_j|\delta}{1-\delta}\opt$ for every $j\in[k_i+1]$. Thus, $f(S_i)\geq (1-\frac{6n\delta}{1-\delta})\, \frac{\opt}{3}$. Finally, if $i^*$ exists we have, $\hat{f}(S_i)+\hat{f}(\{i^*\})\geq \tau_i B$, and therefore, $\alg \geq (1-O(\delta))\frac{\opt}{3}$. Overall, in the presence of a noisy oracle we have approximation guarantee $\frac{\min\{(1-O(\delta))(1-\epsilon),\, 1-\frac{6n\delta}{1-\delta}\}}{3}=(1-\epsilon) \left(1-O(\frac{n\delta}{1-\delta})\right)\, \frac{1}{3}$.

\hfill\Halmos \end{proof}

%
\begin{proof}{Proof of Theorem \ref{resbudget}(ii).}
For each $X$ the algorithm makes $O(\frac{n}{\epsilon}\log n)$ queries and there are $O(n^{1/\epsilon})$ possible sets $X$ leading to $O(\frac{1}{\epsilon}n^{1+\frac{1}{\epsilon}} \log n)$ queries in total.

To show the guarantee we focus on set $S(X)$ when $X\subseteq \opt$ contains the $1/\epsilon$ largest budget elements in \opt\ (or all of \opt\ if its cardinality is small). For this $X$, the budget $b_e$ required by any $e\in \opt\backslash X$ is strictly smaller than $\epsilon B$, otherwise $b(X)>B$.

Consider $\tau_i\in \left[(1-\epsilon)\frac{\opt}{2B},\frac{\opt}{2B}\right]$ and let $S_i$ denote the solution returned by the algorithm with $X$ and threshold $\tau_i$. We show that $f(S_i)\geq (0.5-\epsilon)\, \opt$. The analysis is split in two cases based on how the algorithm terminates.
\smallskip

\noindent \textbf{Case I:} Final Add is not invoked. 
In this case 
every element not in $S_i$ fails the threshold requirement and the algorithm does not remove elements after they are chosen. Similar to the analysis of Algorithm \ref{calg}, we have an interleaved partition 
\[\{O_1,\{s_1\},O_2,\{s_2\},\cdots,\{s_{k_i}\},O_{k_i+1}\},\]
of $\opt\cup S_i$ such that $s_j$ is the $j$th element added to $S_i$. Using Corollary \ref{interleaf}, we have
\[\opt \leq f(S_i) + \sum_{j=1}^{k_i+1} f\left(O_{j} \mid E(j-1)\right).\]
Now, $ f\left(O_{j} \mid E(j-1)\right)\leq \tau_i |O_j|,\, \forall j\in[k_i+1]$. Thus, $\alg\geq f(S_i)\geq  0.5\, \opt$.
\smallskip


\noindent \textbf{Case II:} Final Add is invoked for some element $j$. We claim that the set $S_i$ output by Final Add is such that (i) $B\geq b(S_i)\geq (1-\epsilon)\, B$ and (ii) $f(S_i)\geq \tau_i b(S_i)$. Using (i) and (ii), we have for $\epsilon\in[0,1]$,
\[f(S_i)\geq 0.5\,(1-\epsilon)^2\, \opt\geq \,(0.5-\epsilon)\, \opt.\]
It remains to show (i) and (ii). Let $S_{in}$ denote the set that is input to Final Add and let $X_i=\{e \mid b_e\geq \epsilon B, e\in S_{in}\}$. 
Observe that $X_i \subseteq X$, since every element outside $X$ with budget exceeding $\epsilon B$ is discarded in the beginning. Therefore, $b(X_i)\leq B$ and the set $S_i$ returned by Final Add is feasible (and contains $X_i$). To see the lower bound on $b(S_i)$, let $t$ denote the last element removed from $S_{in}$ by Final Add. We have, $b_t<\epsilon B$ and $b_t+b(S_i)>B$. Thus, $b(S_i)\geq B-\epsilon B$. 

To show (ii) we use Corollary \ref{forE}(ii). 
For simplicity, let $\{1,2,\cdots,s\}$ denote the elements of $S_i$ in submodular order. Recall that $S_i\subset S_{in}$. 
Using Corollary \ref{forE}(ii) with $A=S_{in}$ and sets $E_k=\{k\}$ for $k\in[s]$, we have
\[\sum_{k\in[s]} f(k\mid L_k)\leq f(S_i),\]
where $L_k$ denotes the set of all elements in $S_{in}$ chosen by the algorithm prior to $k$. From the threshold requirement, we have $f(k\mid L_k)\geq \tau_ib_k,\, \forall k\in[s]$. Thus, $f(S_i)\geq \tau_ib(S_i)$. 

In case of a noisy oracle, similar to previous analyses it can be verified that the approximation guarantee is reduced by a multiplicative factor $\left(1-O(\frac{n\delta}{1-\delta})\right)$. 
\hfill\Halmos \end{proof}

\subsection{Matroid Constraint}\label{sec:matr}

\begin{lemma}\label{matrbasic}
Given a matroid $\mathcal{M}$ and an independent set $A$, define $\mathcal{I}_A=\{B\mid A\cup B \in\mathcal{I},\, B\subseteq N\backslash A\}$. For every independent set $A'$ such that $r(A')=r(A\cup A')=r(A)$, we have $\mathcal{I}_A=\mathcal{I}_{A'}$. 
\end{lemma}
\begin{proof}{Proof.}
Consider an arbitrary set $B\in \mathcal{I}_A$. It suffices to show that $B\in \mathcal{I}_{A'}$. There are many ways to show this, in the spirit of submodularity we will prove this by using the submodularity of rank function $r(\cdot)$. 

Note that $r(A\cup B)=|A|+|B|=|A'|+|B|$ and consider the sets $A' \cup B$ and $A\cup A'$. We have from submodularity,
\begin{eqnarray*}
r\left(A'\cup A \cup B\right) + r\left(A'\right)&\leq &r\left(A'\cup B\right) + r\left(A\cup A'\right),\\
r \left(A'\cup A\cup B\right) &\leq &r\left(A'\cup B\right).
\end{eqnarray*}
Therefore, $r(A'\cup B)=r(A'\cup A\cup B)$. Since $r(A' \cup B)\leq |A'|+|B| =r(A\cup B)$, we have that $A'\cup B$ is independent and $B\subseteq N\backslash A'$. 
\hfill\Halmos \end{proof}

\begin{lemma}\label{maximal}
At the end of every iteration, marked by the element $j\in[n]$ that was parsed, the set $S$ maintained by Algorithm \ref{malg} is maximally independent, i.e., $r(S)$ equals the rank of the set $[j]$ that contains all elements parsed so far. 
\end{lemma}
\begin{proof}{Proof.}
Note that the algorithm always maintains an independent set and the lemma is true after the first iteration. For the sake of contradiction, let $j+1$ be the first element in submodular order where the lemma is not true. Let $S_{j}$ denote the set at the end of iteration $j$ (after element $j$ is parsed). By definition of $j$, $S_j\cup\{j+1\}$ is not independent but $r\left( [j+1]\right) =r\left(S_j\right)+1$. Using the fact that $S_j$ is rank maximal subset of $r\left( [j]\right)$, there exists an independent set $S'_j\in[j]$ with the same rank as $S_j$ and such that $S'_j \cup\{j+1\}$ is independent. Applying Lemma \ref{matrbasic} with $A=S_j$ and $A'=S'_j$, we have $\{j+1\}\in\mathcal{I}_{S_j}$, a contradiction.
\hfill\Halmos \end{proof}
\begin{proof}{Proof of Theorem \ref{resmatr}.} Algorithm \ref{malg} parses the ground set once. When parsing an element, the algorithm makes at most $d$ queries, resulting in total $nd$ queries.

To show the performance guarantee. Let $S_j$ denote the set $S$ maintained in the algorithm at the beginning of iteration $j$. Similarly, let $R_j$ denote the set $R$ at the beginning of iteration $j$. Observe that $R_j$ is the set of all elements that were chosen and later swapped out, prior to iteration $j$. Therefore, the set $S_j\cup R_j$ grows monotonically and includes all elements selected by the algorithm prior to $j$. We use $S$ and $R$ to denote the final sets when the algorithm terminates. 
Let $\opt$ denote the optimal set (and value) and \alg\ denote the algorithm output (and value). By monotonicity, $f(\alg \cup \opt)\geq \opt$. 
We will ``essentially" show that $f(\opt \mid \alg)\leq 3\alg$. The main elements that contribute to this upper bound are,
\smallskip

\noindent \textbf{Swap operations:} If an element $j$ is taken out of the set at iteration $t$, then the elements in \opt\ parsed between $j$ and $t$, but never included, may have larger marginal value after the swap. We show that the resulting total increase in marginal value is upper bounded by \alg. 
\smallskip

\noindent \textbf{Rejection of elements:} Any element that is rejected by the algorithm on parsing (due to insufficient marginal value for swapping) may be an element of \opt. Let $j$ denote the index (in submodular order) of such an element. 
Since $j$ is rejected we have that $S_j\cup \{j\}$ contains a circuit $C_j$ and the marginal $f(j\mid S_j\cup R_j)$ is upper bounded by $v_i$ for all $i\in C_j$. 
We show that the total marginal value $\sum_{j\in \opt\backslash (S\cup R)} f(j\mid S_j\cup R_j)$ of elements in $\opt$ rejected by the algorithm is at most 2\alg.
\smallskip

W.l.o.g., we ignore all elements in $N\backslash (\opt\cup S\cup R)$. So let $\hat{N}=\opt\cup S\cup R$ denote our ground set and re-index elements in $\hat{N}$ from $1$ to $|\hat{N}|$ (maintaining submodular order). Similarly, we re-index the sets $S_j, R_j$ so that they continue to denote the set maintained by the algorithm when element $j\in \hat{N}$ is first parsed, for every $j\in\{1,2,\cdots,|\hat{N}|\}$. From monotonicity, $\opt\leq f(\hat{N})$.

Consider an interleaved partition $\{O_\ell,E_\ell\}$ of $\hat{N}$ such that $E(m)=\cup_{\ell\in[m]}E_\ell:= S\cup R$ and $O(1)=\cup_{\ell_\in[m]}O_\ell:= \opt\backslash (S\cup R)$. 
Using Corollary \ref{interleaf},
\begin{eqnarray*}
\opt&\leq &f\left(E(m)\right) +\sum_{\ell\in[m]} f\left(O_\ell \mid E(\ell-1)\right),\\
&\leq &f\left(E(m)\right) +\sum_{j\in \opt\backslash (S\cup R)} f\left(j \mid S_j\cup R_j\right),
\end{eqnarray*}
where the second inequality follows from weak submodular order. Next, we apply Corollary \ref{forE}(i) to upper bound $f(E(m))$ in terms of \alg. Consider an interleaved partition $\{\bar{O}_\ell,\bar{E}_\ell\}_{\ell\in[\bar{m}]}$ of $E(m)$ such that $\bar{E}(\bar{m})=S$ and $\bar{O}(1)=R$. We have,
\begin{eqnarray*}
f\left(E(m)\right) &\leq &\alg + \sum_{\ell\in[\bar{m}]} f\left(\bar{O}_\ell \mid \bar{E}(\ell-1) \cup \bar{O}(1)\backslash \bar{O}(\ell) \right),\\
&= &\alg+ \sum_{j\in R} f\left(j\mid S_j\cup R_j \right).
\end{eqnarray*}
We upper bound $\sum_{j\in R} f\left(j\mid S_j\cup R_j \right)$ by \alg\ and $\sum_{j\in \opt\backslash (S\cup R)} f\left(j \mid S_j\cup R_j\right)$ by $2\alg$. This proves the main claim.
\smallskip

\noindent \textbf{Part I:} $\sum_{j\in R} f\left(j\mid S_j\cup R_j \right)\leq  \alg$. 

\noindent Consider an element $j_1\in R$ that was added to $S_{j_1}$ without swapping out any element. Since $j_1\in R$, there exists an element $j_2$ that later replaced $j$. Inductively, for $t\geq 2$, let $j_t$ denote the $t$th element in the chain of swaps $j_1\to j_2\to\cdots\to j_{t}$. The chain terminates at an element in $S$ and we call this a \emph{swap chain}. From the swap criteria in Algorithm \ref{malg}, we have
\[\sum _{\tau\in[t-1]} f\left(j_{\tau}\mid S_{j_\tau}\cup R_{j_\tau}\right)\leq f\left(j_t\mid R_{j_t}\cup S_{j_t}\right).\] 
Every element in $R$ is part of a unique swap chain 
and each chain has a unique terminal element in $S$. Therefore,
\[\sum_{j\in R} f\left(j\mid S_j\cup R_j \right)\leq \sum_{i\in S} f\left(i\mid S_i\cup R_i \right). \]
Applying Corollary \ref{forE}(ii) with $A=S\cup R$ and sets $E_i=\{i\}$ for $i\in S$, we have
\begin{equation}\label{coroE}
\sum_{i\in S} f\left(i\mid S_i\cup R_i \right) \leq  \alg. 
\end{equation} 

Observe that given a noisy function oracle $\hat{f}$, we have,
\begin{eqnarray*}
\sum_{j\in R} f\left(j\mid S_j\cup R_j \right)&\leq &\sum_{j\in R} \hat{f} \left(j\mid S_j\cup R_j \right) +\frac{O(|R|\delta)}{1-\delta}f(\hat{N}),\\
&\leq &\sum_{i\in S} \hat{f}\left(i\mid S_i\cup R_i \right)+\frac{O(|R|\delta)}{1-\delta}f(\hat{N}),\\
&\leq &\sum_{i\in S} f\left(i\mid S_i\cup R_i \right)+\frac{O(n\delta)}{1-\delta}f(\hat{N}),
\end{eqnarray*}
here the first and third inequalities follow from the definition of noisy oracle and the second inequality follows from the swap criteria. Thus, $\sum_{j\in R} f\left(j\mid S_j\cup R_j \right)\leq \alg + O(\frac{n\delta}{1-\delta})f(\hat{N})$.
\smallskip

\noindent \textbf{Part II:} $ \sum_{j\in \opt\backslash (S\cup R)} f\left(j \mid S_j\cup R_j\right)\leq 2\, \alg$. 

\noindent From Lemma \ref{maximal} we have that, $r\left(\opt\right)\leq r\left(\alg\right)$. Suppose there exists an injection $\phi$ from elements in $\opt\backslash (S\cup R)$ to elements in $S$ such that
\[f\left(j \mid S_j\cup R_j\right) \leq 2 f\left(\phi(j) \mid S_{\phi(j)}\cup R_{\phi(j)} \right), \quad \forall j\in\opt\backslash (S\cup R).\] 
Summing up these inequalities and using \eqref{coroE}, we are done. It remains to show that $\phi$ exists. We do this via a graphical construction inspired by \cite{chekuristream}.

Consider a graph $G$ with vertices given by $\hat{N}$. In order to define the edges recall that every $j\in\opt\backslash (S\cup R)$ is rejected by the algorithm on parsing. Thus, we have a unique circuit $C_j$ where $C_j \backslash \{j\}\subseteq S_j$. 
For our first set of edges we make a directed edge from $j$ to every element in $C_j\backslash\{j\}$ and we do this for all $j\in\opt\backslash (S\cup R)$. Next, for every $j\in R$, let $C_j$ represent the chain that causes $j$ to be swapped out in the algorithm. We create a directed edge from $j$ to every element in $C_j\backslash\{j\}$. Graph $G$ has the following properties,
\begin{enumerate}[(a)]
\item The elements of  $\opt\backslash (S\cup R)$ are \emph{source} vertices with no incoming edges. Elements of $S$ are \emph{sinks} with no outgoing edges.
\item The neighbors of every node in $G$ form a circuit with the node.
\item Given arbitrary vertex $j\in\opt\backslash (S\cup R)$, for every $i$ reachable from $j$ we claim that 
\[f(j\mid S_j\cup R_j)\leq v_i,\]
where $v_i$ corresponds to the value defined in Algorithm \ref{malg}. For neighbors of $j$ the claim follows directly from the swap criterion. Also, for any two neighboring vertices $i',i'' \in S\cup R$, we have $v_{i'}\leq v_{i''}$. The general claim follows by using these inequalities for every edge on the path from $j$ to $i$. 
\end{enumerate}
Using (a) and (b) we apply Lemma \ref{repeat} to obtain an injection $\phi$ from $\opt\backslash (S\cup R)$ to $S$ such that for every $j\in\opt\backslash (S\cup R)$ there exists a path to $\phi(j)\in S$. Then, from (c) we have, $f(j\mid S_j\cup R_j)\leq v_{\phi(j)}$. Recall the notion of a swap chain defined in Part I and let $W(\phi(j))$ denote the set of all preceding elements of the swap chain that terminates at $\phi(j)$. By definition of $v_{\phi(j)}$ and the swap criterion, \[v_{\phi(j)}=f\left(\phi(j)\mid S_{\phi(j)}\cup R_{\phi(j)}\right)+\sum_{i\in W(\phi(j))}f\left(i\mid S_{i}\cup R_{i}\right) \leq 2f\left(\phi(j)\mid S_{\phi(j)}\cup R_{\phi(j)}\right).\]

Given a noisy function oracle $\hat{f}$, by the preceding argument there exists an injection $\phi$ from $\opt\backslash (S\cup R)$ to $S$ such that, $\hat{f}\left(j \mid S_j\cup R_j\right) \leq 2 \hat{f}\left(\phi(j) \mid S_{\phi(j)}\cup R_{\phi(j)} \right), \,\, \forall j\in\opt\backslash (S\cup R)$. Thus, $\sum_{j\in \opt\backslash (S\cup R)} f\left(j \mid S_j\cup R_j\right)\leq   2\,\alg +O(\frac{n\delta}{1-\delta})f(\hat{N})$. Overall, for a noisy oracle we have $\alg \geq (1-O(\frac{n\delta}{1-\delta}))\,0.25\, f(\hat{N}) \geq (1-O(\frac{n\delta}{1-\delta}))\,0.25\, \opt$.
\hfill\Halmos \end{proof}
\begin{lemma}[Lemma 30 in \cite{chekuristream}]\label{repeat}
Let $M = (N,\mathcal{I})$ be a matroid, and $G$ a directed acyclic graph over $N$ such
that for every non-sink vertex $e\in N$, the outgoing neighbors of $e$ form a circuit with $e$ (alternatively, $e$ is in the span of its neighbors). Let $I\in\mathcal{I}$ be an independent set such that no path in $G$ goes from one element in $I$ to another. Then there exists an injection from $I$ to sink vertices in $G$ such that each $e\in I$ maps into an element reachable from $e$.
\end{lemma}
\section{Upper Bound of 0.5}\label{sec:upb}
The following lemma is useful in verifying the (strong) submodular order property for a given function and order. We use it subsequently in the proof of Theorem \ref{impossible} and again in Appendix \ref{sec:proofassort}. 

\begin{lemma}\label{elementwise}
	A monotone subadditive function $f$ is $\pi$ (strongly) submodular ordered if for any two sets $B\subseteq A$ and an element $i$ to the right of $A$, we have 
	$f(i\mid A)\leq f(i\mid B)$.
\end{lemma}
\begin{proof}{Proof.}
	Strong submodular order implies $f(i\mid A)\leq f(i\mid B)$ by definition. To see the reverse direction, consider an arbitrary set $A$, subset $B$ of $A$, and a set $C$ such that $l_{\pi}(C)>r_{\pi}(A)$. Let $C=\{c_1,\cdots,c_k\}$ represent the elements of $C$ in submodular order and define subsets $C_i=\{c_1,\cdots,c_i\}$, $\forall i\in[k-1]$ and $C_0=\{\emptyset\}$. We are given, 
	\[f(c_i\mid A\cup C_{i-1})\leq f(c_i\mid B\cup C_{i-1}), \quad \forall i\in[k].\]
	Therefore,
	\[f(C\mid A)=\sum_{i\in[k]} f(c_i\mid A\cup C_{i-1})\leq \sum_{i\in[k]} f(c_i\mid B\cup C_{i-1})=f(C\mid B),\]
	where we use telescoping marginal values to rewrite $f(C\mid A)$ and $f(C\mid B)$. 
	\hfill\Halmos \end{proof}
\medskip

\begin{proof}{Proof of Theorem \ref{impossible}.}
	At a high level, we construct a family of submodular order functions where the (unique) optimal solution is ``well-hidden", i.e., only a small subset of the optimal solution can be found by polynomial number of queries to the value oracle. By design, the value of every feasible solution that has small overlap with the optimal set is at most half the optimal value. In the following, we use non-negative integers $n_1,n_2,k_1,k_2,r$ and a real valued scalar $\alpha>0$. We specify the exact values of these parameters later on and note that the values will obey the following rules $n_1>k_1>n_2=k_2$, $\alpha=k_1/k_2$ and $r<\alpha k_2=k_1$.
	Let $N$ denote the ground set (size $n$) with partition $\{N_1,N_2\}$. Let $n_i$ denote the cardinality of $N_i$ for $i\in\{1,2\}$. 
	For sets $S\subseteq N_1$, we define $f(S)=\min\{|S|,2k_1\}$. Under this definition, we have a monotone submodular function on the ground set $N_1$. Notice that, for all sets $S\subseteq N_1$ with $|S|\leq 2k_1$, the function value equals $|S|$. Further, subsets of $N_1$ of the same size have identical value. 
	For sets $S\subseteq N_2$, we define $f(S)=\alpha|S|$. Recall that $n_2=k_2$ and by definition, $f(N_2)=\alpha k_2=k_1$. Our overall ground set is given by $N_1\cup N_2$ so it remains to define function value on sets that intersect both $N_1$ and $N_2$.
	
	For a special set $A_1\subseteq N_1$ with $|A_1|=k_1$, we define 
	\[f(A_1\cup S_2):=k_1 + \alpha |S_2|-r\frac{|S_2|}{k_2}\quad \forall S_2\subseteq N_2.\]
	Recall that $r< \alpha k_2=k_1$ thus, $f(A_1\cup S_2)>\max\{f(A_1),f(S_2)\}$. More generally, for $B_1\subseteq A_1$ and $S_2\subseteq N_2$, 
	\[f(S_2\mid B_1):=f(S_2)\left(1-\frac{\min\{r,f(B_1)\}}{k_1}\right)= |S_2|\left(\alpha-\frac{\min\{r,|B_1|\}}{k_2}\right).\] 
	Next, for arbitrary sets $S_i\subseteq N_i$, 
	\begin{eqnarray}\label{gendef}
		f(S_2\mid S_1)&:= &f(S_2)\left(1-\min\left\{1,\frac{|S_1\backslash A_1|+\min\{r,|S_1\cap A_1|\}}{k_1}\right\} \right).
	\end{eqnarray}
	Observe that for $|S_1\cap A_1|\leq r$, this simplifies to,
	\[f(S_2\mid S_1)= f(S_2)\left(1-\min\left\{1,\frac{|S_1|}{k_1}\right\} \right).\]
	By definition, we have that for every element $e\in A_1$, the marginal $f(e\mid N_2)=0$. However, given a set $B_1\subset A_1$ with $|B_1|>r$ and $e\in A_1\backslash B_1$, we have $f(e\mid B_1\cup N_2)=1$, Therefore, $f$ is not a submodular function. We claim that it is monotone subadditive and has a strong submodular order. 
	
	To see monotonicity, consider function value in the form $f(S_1)+f(S_2\mid S_1)$. If an element is added to $S_2$, this sum does not decrease. Now, if $|S_1|\geq 2k_1$, then $f(S_2|S_1)=0$ and we also have monotonicity w.r.t.\ elements added to $S_1$. Finally, when $|S_1|<2k_1$, adding an element to $S_1$ increases $f(S_1)$ by 1 and can decrease $f(S_2\mid S_1)$ by at most $\frac{f(S_2)}{\alpha k_2}\leq 1$. To see subadditivity, observe that $f(S_2)\geq f(S_2\mid S_1)$. 
	
	Fix some order $\pi$ where $N_1$ is entirely to the left of $N_2$. We show that $\pi$ is a strong submodular order for $f$. Consider sets $B\subseteq A$ and an element $i$ to the right of $A$. Using Lemma \ref{elementwise}, it suffices to show that $f(i\mid A)\leq f(i\mid B)$. We establish this in cases. 
	\smallskip
	
	\noindent \textbf{Case I}: $A\cup \{i\}\subseteq N_j$ for $j\in\{1,2\}$. Note that $A\subseteq N_j$ implies $B\subseteq N_j$. The inequality follows from the 
	fact that $f$ is monotone and submodular when restricted entirely to either $N_1$ or $N_2$. 
	\smallskip
	
	\noindent \textbf{Case II}: $A\subseteq N_1$ and $i\in N_2$. Since $B\subseteq A$, we have, $|B\backslash A_1|\leq |A\backslash A_1|$ and $|B\cap A_1|\leq |A\cap A_1|$. Then from \eqref{gendef}, we have that $f(i|B)\leq f(i|A)$ by definition.
	
	\smallskip
	
	\noindent \textbf{Case III}: $A$ intersects both $N_1$ and $N_2$. Since $i$ is to the right of $A$ in order $\pi$, we have that $i\in N_2$. In this case, we claim that, 
	\begin{equation}\label{single}
		f(i\mid B+e)\leq f(i\mid B),\quad \forall i\in N_2, e\not\in B+i.
	\end{equation} Applying \eqref{single} repeatedly gives us the desired. To see \eqref{single}, consider two cases. First, suppose $e\in N_2$. In this case, observe that the submodularity of $f$ on $N_2$ gives us the desired. In the second case, $e\in N_1$, and from \eqref{gendef}, we see that the marginal value of $i$ does not increase. This establishes the submodular order property for $f$. 
	
	
	
	Now, consider the problem of maximizing $f$ subject to cardinality constraint $k:= k_1+k_2$. Observe that the set $A_1\cup N_2$ is the unique optimum with value $2k_1-r$. Further, every feasible set with at most $r$ elements from $A_1$, has value at most $k$. Choosing $k_1=n_1^{0.5-\delta}$ (i.e., $o(\sqrt{n_1})$), and $n_2=k_2=r=k_1^{1-\delta}$ (i.e., $o(k_1)$) for some small $\delta>0$, we have that as $k_1\to +\infty$, 
	\begin{equation*}
		k=k_1+k_2=k_1+o(k_1) <(0.5+\epsilon)\, \opt,
	\end{equation*}
	for any constant $\epsilon>0$. Therefore, an algorithm has approximation guarantee strictly greater than $0.5$ only if it outputs a set with more than $r$ elements from $A_1$. We generate $A_1$ by sampling $k_1$ distinct elements from $N_1$ uniformly at random and show that the random set $A_1$ is ``well hidden" from polynomial query algorithms. Formally, we show that when $A_1$ is generated uniformly randomly, any given polynomial query algorithm will fail to find a feasible set with more than $r$ elements of $A_1$ with probability approaching 1 in the limit of $k_1\to \infty$. This proves the main claim. 
	
	W.l.o.g., we consider algorithms that query only marginal values $f(X\mid Y)$ for sets $X,Y\subseteq N_i$ for $i\in\{1,2\}$. A query in any other form can be converted to a constant number of such queries. We call a set $S$ \emph{good}, if 
	\[S\subseteq N_1,\,\, |S|<2k_1, \text{ and  }|S\cap A_1|>r.\]
	If at least one of the sets $X$ or $Y$ is good, we call $f(X|Y)$ a \emph{good} query. All other queries are called \emph{bad}. To show that $A_1$ is well hidden, we show that, (i) 
	Unless a query $f(X|Y)$ is good, we do not get any useful information about finding a good set. (ii) A polynomial (query) algorithm fails to make a good query with probability approaching 1 (as $k_1\to +\infty$). Combining (i) and (ii), we have that a polynomial algorithm will fail to find any good set with probability approaching 1. 
	
	To show (i), notice that queries with $X\cup Y\subseteq N_i$ give no information (on finding a good set). Thus, it suffices to consider queries with $X\subseteq N_i$ and $Y\subseteq N_{-i}$, where $\{i,-i\}=\{1,2\}$. We consider all possible bad queries through the following cases.
	\smallskip
	
	\noindent \textbf{Case I:} $X\subseteq N_1$, $Y \subseteq N_2$, and $|X|\geq 2k_1$. In this case, we call $X$ a \emph{large} set. 
	Observe that for any given $Y\subseteq N_2$, we have $f(X\cup Y)=2k_1$ for all large sets $X$. 
	In this case the marginal value $f(X\mid Y)$ is identical for every $X$ and we obtain no useful information. 
	\smallskip
	
	\noindent \textbf{Case II:} $Y\subseteq N_1$, $X\subseteq N_2$, and $|Y|\geq 2k_1$. In this case, observe that $f(X\mid Y)=0$ for every $X\subseteq N_2$ and we obtain no useful information. 
	\smallskip
	
	\noindent \textbf{Case III:} $|X\cap A_1|\leq r$ and $|Y\cap A_1|\leq r$. 
	First, for $Y\subseteq N_2$, notice that every set $X\subseteq N_1$ with $|X\cap A_1|\leq r$ has marginal value $f(X\mid Y)=f(X)+f(Y\mid X)-f(Y)$, which only depends on the cardinality of $X$ and $Y$. Similarly, given $Y\subseteq N_1$ with $|Y\cap A_1|\leq r$, every set  $X\subseteq N_2$ of a given cardinality has the same marginal value $f(X\mid Y)$. Thus, we find no useful information on good sets. 
	
	Therefore, given a collection of bad queries $\{f(X_i\mid Y_i)\}$, 
	we do not obtain any useful information on finding a good set. 
	Since a query where at least one of $X$ or $Y$ is a large set gives no useful information on good sets, to show (ii), it suffices to consider algorithms that only make queries $f(X|Y)$ such that, $\max\{|X|,|Y|\}<2k_1$. We call a set with less than $2k_1$ elements \emph{small}.  
	
	First, consider the special case of deterministic algorithms. 
	Consider an arbitrary \emph{small} set $Z\subseteq N_1$. 
	Since $A_1$ is selected uniformly randomly, using a standard concentration bounds for negatively correlated random variables we have that the probability $P[|Z\cap A_1|>r]$, is exponentially small in $n_1$ 
	(see Lemma \ref{chernoff} for a proof). 
	Given polynomially many small sets $\{Z_i\}_{i\in[n']}$, 
	using the union bound 
	we have an exponentially small probability of the event that $|Z_i\cap A_1|>r$ for some $i\in[n']$. Thus, a deterministic algorithm with total number of queries polynomial in $n_1$ will, with probability arbitrarily close to 1, fail to make a good query. 
	Given a randomized algorithm, we condition on its random seed (that is independent of randomness in our instance), to obtain a deterministic algorithm. For a polynomial query algorithm, the conditionally deterministic algorithm must make polynomial number of queries. Thus, it fails to make a good query with probability 1 (asmptotically). 
	Unconditioning on the random seed then gives us the desired bound for every (randomized) polynomial query algorithm. 
	\hfill\Halmos \end{proof}
\begin{lemma}\label{chernoff}
	Consider a ground set $N$ of $n$ elements. Let $A$ be a subset of $k=\floor{n^{0.5-\delta}}$ elements chosen uniformly at random from $N$ for some small $\delta>0$. Then, for any $S$ with cardinality at most $2k$ and $r=k^{1-\delta}$, we have
	\[P\left[|S\cap A|>r\right]
	<e^{-\Theta(n^{1-\delta})}.\]
\end{lemma}
\begin{proof}{Proof.}
	Let $S=\{1,2,\cdots,s\}$, where $s\leq 2k$. For every $i\in[s]$, define indicator random variable $X_i$ that is 1 if $i\in A$ and 0 otherwise. Note that $P[X_i=1]=k/n$ for every $i\in[s]$. We claim that these random variables are negatively correlated such that for every $\hat{S}\subseteq S$,
	\[P\left[ \bigwedge_{j\in \hat{S}} X_j=1 \right]\leq \prod_{j\in \hat{S}} P\left[X_j=1\right]. \]
	We prove this by induction. The claim is true for singleton sets $\hat{S}$. We assume the claim holds for all sets $\hat{S}$ with cardinality $\hat{k}$ or less. Consider a set $\hat{S}+e$ of size $\hat{k}+1$ with $e\in S\backslash \hat{S}$. We have,
	\[P\left[X_e=1 \, \Big|\, \bigwedge_{j\in \hat{S}} X_j=1 \right]= \frac{k-\hat{k}}{n-\hat{k}} \leq P[X_e=1],\]
	where the inequality follows from $k<n$. Now, a direct application of the generalized Chernoff bound given in Theorem 3.4 in \cite{panconesi1997randomized}, gives us
	\[P\left[|S\cap A|>(1+\epsilon)2n^{-2\delta} \right]
	<e^{-\frac{\Theta(\epsilon^2)}{n^{2\delta}}},\]
	for $\epsilon>0$. Substituting $\epsilon=O(n^{0.5+0.5\delta})$ completes the proof.
	\hfill\Halmos \end{proof}

\section{Conclusion}\label{sec:conclusion}
We introduced notions of weak and strong submodular order functions and explored the landscape of constrained maximization problems under this structure. We showed that in the value oracle model, no polynomial algorithm can achieve a guarantee better than $0.5$ for maximizing function value subject to a cardinality constraint. We proposed algorithms that achieve this best possible guarantee under cardinality and more generally, budget constraint. We also gave a 0.25 approximation for the problem under matroid constraint.
Applying these results, we obtained improved (and first constant factor) guarantees, with unified algorithms, for several constrained assortment optimization problems. 
We also observed an intriguing algorithmic connection of submodular order maximization with the problem of streaming submodular optimization. The connection raises some interesting questions about the interplay between structure of set functions and computational considerations such as limited memory. 
\medskip

\noindent \textbf{Directions for Further Work and Beyond Submodular Order} 
\begin{itemize}
\item What is the best possible approximation guarantee under matroid constraint? 
\item  \textbf{Multiple submodular orders:}  Consider an order $\pi$ and its reverse order $\pi_r$. Does a function with submodular order along both $\pi$ and $\pi_r$ admit better approximations? Note that the upper bound of $0.5$ in Theorem \ref{impossible} applies to functions with multiple submodular orders but not to functions where an order and its reverse are both submodular orders.  
\item \textbf{Unknown submodular order:} Given a function $f$ and the knowledge that it has a submodular order, can we get interesting guarantees for optimization without knowing the order? Are there natural conditions under which the order could be computed using only value queries? What is the computational complexity of verifying that a given order is a submodular order or certifying the inexistence of (weak) submodular order? Note that testability problems are hard and not completely understood even for classic submodular functions \cite{seshadhri2014submodularity}.

Recall that knowing the submodular order is not necessary for obtaining algorithmic guarantees. In particular, our algorithm for Markov choice model develops a partial (submodular) order by repeatedly solving an unconstrained version of the problem on smaller and smaller ground sets. Generalizing the scope of such an idea idea is another interesting direction for future work. Recall that the impossibility result given in Theorem \ref{mor} only applies to functions that do not have any (strong or weak) submodular order.  
\item Are there other applications where objective is not submodular but has submodular order?
\end{itemize}


\section*{Acknowledgements}

The author thanks Ali Aouad, Omar El-Housni, Vineet Goyal, and Danny Segev for helpful discussions and comments on early drafts of this work. 

{\small
\bibliographystyle{informs2014.bst}
\bibliography{assortbib}}
\begin{APPENDICES}
{\color{black}	\section{Performance of Multi-element Extensions of Classic Algorithms}\label{appx:multigrd}
\textbf{Greedy Algorithm:}	Consider the greedy algorithm that iteratively builds a solution of cardinality at most $k$ by adding a set of up to $m$ elements in each iteration, i.e., given current set $S$ the algorithm will add set
\[X=\argmax_{A\subseteq N\backslash S, |A|\leq \min\{k-|S|, m\}} f(A\cup S).\]
For $m=1$, this is the standard greedy algorithm. The algorithm is efficient (makes polynomially many queries to oracle for $f$) as long as $m$ is some constant (independent of $k,n$).

\emph{Generalization of Example 1:} Consider a ground set $\{1,\cdots,2k+m\}$. Let singleton values $f(\{e\})$ equal $1$ for all \emph{good} elements $e\in\{1,\cdots,k\}$, equal $\epsilon$ for all \emph{poor} elements $e\in\{k+1,\cdots,2k\}$ and finally, let $f(\{e\})=1+\epsilon$ for all elements $e\in\{2k+1,\cdots, 2k+m\}$. Let the first $2k$ elements be modular, i.e., the value of a set $S=S_1\cup S_2$ with subset $S_1$ of good elements and $S_2$ of poor elements is simply $|S_1|+\epsilon|S_2|$. Finally, let $f(S_1\cup S_2\cup\{2k+1,\cdots, 2k+m\})=\max\{|S_1|,m(1+\epsilon)\}+\epsilon|S_2|$ for every subset $S_1$ of good elements and $S_2$ of poor elements. It can be verified that this function is monotone subadditive and the natural indexing $\{1,\cdots,2k+m\}$ is a submodular order. Notice that the set of all good elements has value $k$ and this is the optimal set of cardinality $k$ for every $k\geq 2$. However, the greedy algorithm would first pick the set $\{2k+1,\cdots, 2k+m\}$ and will subsequently pick $k-m$ poor elements, resulting in a total value of $m+\epsilon\, k$. For $\epsilon\to 0$ and $k,n\to+\infty$, this is only a trivial $1/k$ approximation of the optimal value.

\textbf{Local Search:} In the example above, performing local search by swapping up to $m$ elements at a time (with the objective of improving function value) will not find any improvement on the set $\{k+m,\cdots,2k+m\}$. This set is a local maxima with value $1/k$ of the optimal. 

}

\section{Proof of Main Results for Assortment Optimization}\label{sec:proofassort}

In this section, we establish the results for assortment optimization discussed in Section \ref{sec:asstresult}. Recall that to obtain monotonicity we transform the revenue objective $R_{\phi}(S)$ in assortment optimization to $f_{\phi}(S)=\max_{X\subseteq S} R_{\phi}(S)$. We start by proving the subadditivity of this function for substitutable choice models. Recall that a choice model is 
substitutable if $\phi(i,S)\geq \phi\left(i,S\cup\{j\}\right),\, \forall i\in S, j\not\in S$. 
Recall that both MNL and Markov models are substitutable \citep{blanchet2016markov}. 

\begin{lemma}\label{asst:subs}
For any substitutable choice model $\phi$, the function $f_{\phi}$ is monotone subadditive.
\end{lemma}
\begin{proof}{Proof.}
Monotonicity follows by definition. To see subadditivity, consider any two sets $S_1$ and $S_2$ and let $Y= \argmax_{X\subseteq S_1\cup S_2} R_\phi(X)$. Using substitutability,
\begin{eqnarray*}
f_\phi(S_1\cup S_2)=R_\phi(Y)&& = \sum_{i\in S_1\cap Y}r_i\phi(i,Y) + \sum_{i\in S_2\cap Y}r_i\phi(i,Y) \\
&& \leq \sum_{i\in S_1\cap Y}r_i\phi(i,S_1\cap Y) + \sum_{i\in S_2\cap Y}r_i\phi(i,S_2\cap Y)\\
&&\leq f_{\phi}(S_1)+f_\phi(S_2).
\end{eqnarray*}
\hfill\Halmos \end{proof}
Next, we prove the submodular order property for MNL choice models and obtain new results for the joint customization and assortment optimization problem. 
\subsection{New Results in the MNL Model}\label{sec:asstmnl}
We start by stating a well known property of MNL models. A proof is provided for completeness.

\begin{lemma}\label{mnlprop}
Given an MNL model $\phi$, an assortment $S$, and an element $i\not\in S$, we have $R_{\phi}(i\mid S)\geq 0\,  \Leftrightarrow\,  r_i\geq R_\phi(S)$. Further, if $r_i\geq R_{\phi}(S)$ then $r_i\geq R_\phi(S\cup\{i\})$.
\end{lemma} 
\begin{proof}{Proof.} We drop $\phi$ from notation for convenience. Let $v(S)=v_0+\sum_{j\in S} v_j$. The following chain proves the lemma,
\begin{eqnarray*}
R(i\mid S)= \frac{r_iv_i+R(S)v(S)}{v_i+ v(S)} - R(S)\,=\, \frac{r_iv_i-R(S)v_i}{v_i+ v(S)}\, =\, v_i\frac{\left(r_i-R(S)\right)}{v_i+ v(S)}\,\leq\, r_i-R(S).
\end{eqnarray*}
\hfill\Halmos \end{proof}


\begin{lemma}\label{mnlorder}
The order $\pi$ given by sorting products in descending order of price $r_i$, breaking ties arbitrarily, is a strong submodular order under MNL choice.
\end{lemma} 
\begin{proof}{Proof.}
We omit reference to choice model $\phi$ from the subscript. Index elements in $N$ in descending order of prices $r_1\geq r_2 \cdots \geq r_n$, breaking ties arbitrarily. Call this ordering $\pi$ and consider two 
sets $A,B$ with $B\subseteq A$.  
Using Lemma \ref{elementwise}, it suffices to consider an arbitrary element $i$ to the right of $A$ and show $f(i\mid A)\leq f(i\mid B)$.

First, consider the case $r_i<f(A)$. From Lemma \ref{mnlprop} we have $f(i\mid A)=0$ and non-negativity of $f$ gives us the desired. Now, let $r_i\geq f(A)$. Since $r_i$ is the lowest price element in $A+i$, repeatedly applying (both parts of) Lemma \ref{mnlprop} we have that $f(A)=R(A)$ and $f(A+i)=R(A+i)$, i.e., the optimal unconstrained assortment on sets $A$ and $A+i$ includes all elements. 
Similarly, using $f(B)\leq f(A)$ we have, $f(B)=R(B)$ and $f(B+i)=R(B+i)$. 
Thus, it only remains to show that $R(i\mid A)\leq R(i\mid B)$. Letting $v(S)=v_0+\sum_{j\in S} v_i$, we have
\begin{eqnarray*}
R(i\mid A) 
=\frac{v_i\left(r_i-R(A)\right)}{v_i+ v(A)}\, \leq\, 
\frac{v_i\left(r_i-R(B)\right)}{v_i+ v(B)}=R(i\mid B).
\end{eqnarray*} 
Notice that the argument does not work if $r_i$ is not the smallest revenue element in $A$. So,  in general, orders other than non-increasing price order are not submodular orders for MNL.  
\hfill\Halmos \end{proof}
\smallskip

\noindent \emph{Remark:} 
\cite{aouad2018greedy} introduce and show a \emph{restricted} submodularity for MNL model. This notion is equivalent to submodularity enforced over ``small" sets. Formally, $f(C\mid A)\leq f(C\mid B)$ for all sets $B\subseteq A$ and $|A\cup C|\leq s$, for some $s\geq2$. This weakening of submodularity is quite distinct from submodular order property. 

Now, recall that in the joint customization and assortment optimization problem we seek a selection $S$ of at most $k$ products to maximize,
\[\sum_{j\in[m]}\alpha_j \max_{X\subseteq S} R_{\phi_j}(X)\]
Next, we prove our main result (Theorem \ref{mnlresult}) for this problem. 
\begin{proof}{Proof of Theorem \ref{mnlresult}.}
Using the definition of $f_{\phi_j}$ we reformulate the problem as, \[\max_{S,|S|\leq k}\sum_{j\in[m]}\alpha_j f_{\phi_j}(S).\] 
Using Lemma \ref{asst:subs} and Lemma \ref{mnlorder}, every $f_{\phi_j}$ is monotone subadditive and has (strong) submodular order in the direction of descending prices. From Lemma \ref{closed}, the function $F:=\sum_{j\in[m]} \alpha_j f_{\phi_i}(S)$ is also monotone subadditive and has (strong) submodular order in the direction of descending product prices. Further, the value oracle for $F$ can be implemented efficiently with runtime linear in $m$. Thus, a straightforward application of results constrained submodular order maximization gives us the desired (including generalizations to budget and matroid constraint).
\hfill\Halmos \end{proof}

\subsubsection{Joint Customization, Pricing, and Optimization.}\label{sec:asstmnlpricing}
Consider a setting where the prices of products are not fixed and we can customize both the assortment and product prices for each customer type. Let $P=\{p_1,\cdots,p_r\}$ denote the of possible product prices. Given an assortment $S$ and price vector $\mb{r}_S=(r_i)_{i\in S}\in P^{S}$, the probability that type $j$ customer chooses item $i$ is specified by an MNL model $\phi_j(i,S,\mb{r})$. The joint optimization problem is stated as follows,
\begin{equation}\label{customprice}
\max_{S, |S|\leq k}\sum_{j\in[m]}\left( \alpha_j \,\max_{X\subseteq S,\mb{r}_S} \sum_{i\in S}r_i \phi_j(i,S,\mb{r})\right).
\end{equation}
\begin{corollary}\label{mnlcorollar}
The joint customization, pricing, and assortment optimization problem under MNL choice has a 0.25 approximation. 
\end{corollary}
\begin{proof}{Proof.}
We prove the corollary by reducing the problem to an instance of matroid constrained joint assortment optimization and customization problem (with fixed prices). For every type $j\in[m]$, recall that the MNL model $\phi_j$ is specified by parameters $v_{j,i,p_r}$ for every $i\in N,p_r\in P$. Consider the expanded ground set $N_{m,P}=\{(j,i,p_r)\mid j\in[m],i\in N,p_r\in P\}$. For each type $j\in[m]$, define MNL model $\hat{\phi}_j$ on ground set $N_{m,P}$ such that $\hat{v}_{j',i,p_r}=0$ for every $j'\neq j$ and $\hat{v}_{j,i,p_r}=v_{j,i,p_r}$. 
Consider the joint customization and assortment optimization problem (with fixed prices) on ground set $N_{m,P}$ and choice models $\hat{\phi}_j$ for $j\in[m]$, subject to a matroid constraint that ensures that: (i) we pick at most $k$ products and (ii) for every $j\in[m]$ and $i\in[N]$, we pick at most one product from the set $\{(j,i,p_r)\mid p_r\in P\}$. It can be verified that this is an equivalent reformulation of the original pricing problem problem. 
\hfill\Halmos \end{proof}

\subsection{New Results for Markov Choice and Beyond}\label{asst:markov}

%

\subsubsection{Proof of Theorem \ref{compatible}.}\label{sec:proofofcompat}
\begin{proof}{} Theorem 9 has two main claims. The first claim is that when $\phi$ is compatible Algorithm \ref{framework} has the same guarantee as the underlying algorithm $\mathcal{A}$. The second claim is that the Markov choice model is compatible. We prove these claims separately. In Part A, we show approximation guarantees for Algorithm \ref{framework} under the assumption of a compatible model. Part B, which can be read independently, establishes the compatibility of Markov model. 

\subsubsection*{A - Analyzing Algorithm \ref{framework}.} 
At a high level, the analysis has two parts. In Part A-I, we introduce the notion of piece-wise submodular order and show that this notion is sufficient to generalize the guarantees obtained for submodular order functions. In Part A-II, we show that compatible choice models exhibit the required notion of piece-wise submodularity.
\smallskip

\noindent {\textbf{Part A-I}:} To introduce the notion of a piece-wise submodular order, consider an algorithm $\mathcal{A}$ and fix a parameter setting $\gamma\in \Gamma_{\mathcal{A}}$. Suppose that the \textbf{while} loop in Algorithm \ref{framework} runs in $p$ phases, each defined by a call to $U_{\phi}$. With $M,\hat{N}$ as defined in the algorithm, consider the beginning of phase $i\in[p]$ and define, 
\[\text{(Beginning of phase $i$) }\qquad N_i:=\hat{N}\backslash M.\]
Note that $N_i$ contains all elements of $\hat{N}$ that have not been passed to $\mathcal{A}_{\gamma}$ prior to phase $i$. At the end of phase $i$, let $M_i$ be the set of new elements picked by $\mathcal{A}_{\gamma}$, i.e.,  
\[\text{(End of phase $i$) }\qquad M_i=M\cap N_i.\]
At the end of the last phase, let $N_{p+1}$ denote the set of elements that were never passed to $\mathcal{A}_{\gamma}$. Let $M_{p+1}=\emptyset$. Notice that $\{N_i\}_{i\in[p+1]}$ is a partition of the original ground set $N$. Let $\pi$ denote an order such that
\[r_{\pi}(N_i)<l_{\pi}(N_{i+1}),\quad \forall i\in[p]. \] 

\noindent \textbf{Proper sets:} 
We say that set $A$ is $k$-\emph{proper} if $A\cap N_i\subseteq M_i$ for $i\leq k-1$ and $A\cap N_i=\emptyset$ for all $i>k$. Using this definition, every subset of $N_k$ is $k$-proper and given a $k$-proper set $A$, all subsets of $A$ (including the empty set) are also $k$-proper. 
Finally, every subset of $\cup_{i\in[p]} M_i$ is $p+1$ proper. We use the term proper set to refer to a set that is $k$-proper for some $k\in[p+1]$.
\smallskip

\noindent \textbf{Piece-wise submodular order:} 
Order $\pi$ is a $\{N_i,M_i\}_{i\in[p+1]}$ piece submodular order if for every 
\emph{proper} set $A$, set $B\subseteq A$, and $C$ to the right of $A$, we have $f(C\mid A)\leq f(C\mid B)$.

Note that for $p=0$ we recover the notion of submodular order. While this notion is weaker than submodular order, we have the following crucial upper bounds in the same vein as Corollary \eqref{interleaf} and Corollary \eqref{forE}. 

\begin{lemma}\label{interleafextend}
Consider a monotone subadditive function $f$ on ground set $N$ and $\{N_i,M_i\}_{i\in[p+1]}$ piece submodular order $\pi$. For every set $A$ that has an interleaved partitioned $\{O_\ell,E_\ell\}_{\ell\in[m]}$ such that $E(m) \subseteq \cup_{i\in[p+1]} M_i$, we have
\[f(A)\leq f(E(m)) +  \sum_{\ell} f(O_\ell\mid E(\ell-1)). \]
\end{lemma} 

\begin{lemma}\label{forEextend}
Consider a monotone subadditive function $f$ on ground set $N$ and $\{N_i,M_i\}_{i\in[p+1]}$ piece submodular order $\pi$. We have that $f$ is $\pi$-submodular ordered on the restricted ground set $\cup_{i\in[p+1]} M_i$. 
Thus, given a set $A\subseteq \cup_{i\in[p+1]} M_i$, with interleaved partitioned $\{O_\ell,E_\ell\}_{\ell\in[m]}$ and sets $L_\ell=\{e\in A\mid \pi(e)<l_{\pi}(E_\ell)\},\, \forall \ell\in[m]$, we have 
\begin{enumerate}[(i)]
\item $f(A)\leq f\left(E(m)\right) + \sum_{\ell\in[m]} f\left(O_\ell \mid E(\ell-1)\cup O(1)\backslash O(\ell) \right),$
\item $\sum_{\ell\in[m]} f\left(E_\ell \mid L_{\ell}  \right) \leq f\left(E(m) \right).$ 
\end{enumerate}
\end{lemma}
Appendix \ref{appx:interextend} gives a proof of Lemma \ref{interleafextend}. Lemma \ref{forEextend} is proved in Appendix \ref{appx:forEextend}. Interestingly, (the more general) Lemma \ref{mainleaf} is not valid for piece-wise submodular order. Nonetheless, we recover all the guarantees for submodular order functions as the analysis of \emph{every} algorithm $\mathcal{A}$ in this paper 
relies only on Corollary \eqref{interleaf} and Corollary \eqref{forE} (for sets $E(m)$), in addition to monotonicity and subadditivity. The following lemma formalizes this observation. See Appendix \ref{appx:pieceextend} for a proof. 

\begin{lemma}\label{pieceextend}
Consider an instance of constrained assortment optimization for a compatible choice model $\phi$. Suppose we execute Algorithm \ref{framework} with $\mathcal{A}$ given by an appropriate algorithm out of Algorithms \ref{calg}, \ref{bcalg}, \ref{5balg} and \ref{malg}. If for every parameter setting $\gamma\in \Gamma_\mathcal{A}$, the order $\pi$ and sets $\{N_i,M_i\}_{i\in[p+1]}$ generated by Algorithm \ref{framework} are such that $\pi$ is a $\{N_i,M_i\}_{i\in[p+1]}$ piece submodular order then, Algorithm \ref{framework} retains the guarantee of $\mathcal{A}$ for submodular order maximization.
\end{lemma}

This concludes the first part of the analysis. We have shown that if $f_{\phi}$ satisfies a more relaxed notion of piece-wise submodular order then, Algorithm \ref{framework} has the same guarantee as the underlying algorithm $\mathcal{A}$ (when $\mathcal{A}$ is used for functions with submodular order). 
\medskip

\noindent \textbf{Part A-II}: In this part we show that compatibility implies the piece-wise submodular order property. Consider a compatible choice model with optimal unconstrained assortment $S$. Recall that a choice model is compatible if, 
\begin{eqnarray*}
&R(X\mid Z)\geq 0  
\qquad\qquad &\forall X\subseteq S,\, Z\subseteq N\qquad\qquad \eqref{prop2}\\
&R(Z\mid X)\leq R(Z\mid Y)\qquad\qquad &\forall Y\subseteq X\subseteq S,\, \, Z\subseteq N\qquad \eqref{prop1}
\end{eqnarray*}
Consider sets $A,B,$ and $C$ such that $A$ is a $k+1$-proper set for some $k\geq0$, $B\subseteq A$, and $C$ is to the right of $A$ in order $\pi$. We need to show that, $f(C\mid A)\leq f(C\mid B)$. First, using properties \eqref{prop2} and \eqref{prop1}, we show that,
\begin{enumerate}[(i)]
\item $f(A\cup C)=\max_{X\subseteq C} R(A\cup X)$.
\item $R(C\mid A)\leq R(C\mid B)$.
\end{enumerate}   

To prove (i) and (ii) for $k+1$ proper sets, define $N^-_k=\cup_{i\leq k}N_i$, $N^+_k=N\backslash N^-_k$, and $M^-_k=\cup_{i\leq k} M_i$. We claim that the assortment $N_{k+1}\cup M^-_k$ is an optimal unconstrained assortment on ground set $N^+_k\cup M^-_k$. Since any $k+1$ proper set is a subset of this assortment, using properties \eqref{prop2} and \eqref{prop1} of compatible choice models we have (i) and (ii) as desired. We prove $N_{k+1}\cup M^-_k$ is optimal on ground set $N^+_k\cup M^-_k$ by induction. For $k=0$, $M^-_0=\emptyset$ and $N_1$ is an optimal assortment on $N$ by definition. Suppose the claim holds for $k\in\{0,\cdots,h\}$. 
Then, $S:=N_{h+1}\cup M^-_{h}$ is an optimal assortment on ground set $N^+_{h}\cup M^-_{h}$. Using \eqref{prop2} with $A:=M^-_{h+1}\subseteq S$ and $C:=N^+_{h+1}\cup M^-_{h+1}$, we have that the assortment $N_{h+2}\cup M^-_{h+1}$ is optimal on ground set $N^+_{h+1}\cup M^-_{h+1}$. This completes the induction and the proof of (i) and (ii).

%

The main claim now follows from (i) and (ii). From (i) we have, $f(C\mid A)=\max_{X\subseteq C} R(X\mid A)$ and $f(C\mid B)=\max_{X\subseteq C} R(X\mid B)$ (since $B\subseteq A$ is also a proper set). From (ii) we get, $R(X\mid A)\leq R(X\mid B)$ for every $X$ to the right of $A$. To complete the proof we note that $\max_{X\subseteq C} R(X\mid A)\leq \max_{X\subseteq C} R(X\mid B)$.

\subsubsection*{B - Proof of Compatibility of Markov Model.} 
We omit $\phi$ from subscript for convenience. Let us with some necessary notation and properties of Markov model.
Let $P(i\prec Y)$ denote the probability that $i$ is visited before any element in $Y\cup\{0\}\backslash\{i\}$ in the markov chain, here 0 is the outside option. Let $P_j(i\prec Y)$ denote the probability of $i$ being visited before $Y\cup\{0\}\backslash\{i\}$ when the traversal starts at $j$. 
\cite{desir} introduced an important notion of externality-adjustment, where given disjoint sets $X$ and $Y$ they define,
\[R^X(Y)=R(X\cup Y)-R(X).\]
In the following, we summarize properties of this adjustment and some useful lemmas shown in \cite{desir}:
\begin{enumerate}[(i)]
\item  \textbf{$X$-adjusted markov chain:} $R^X(Y)$ ($=R(Y\mid X)$) is the revenue of set $Y$ in a markov chain where the \emph{reduced} price of every $i\in Y$ is,
\[r^X_i= r_i -\sum_{j\in X}P_{i}(j\prec X) r_j, \]
and reduced prices $r^X_j=0$ for every $j\in X$. Transition probabilities are adjusted so that $\rho^X_{j0}=1$ for $j\in X$ (Lemma 4 and Figure 1 in \cite{desir}). Therefore,  \[R^X(i)=R(i\mid X)=P(i\prec X) r^X_i.\]  
\item \textbf{Composition of adjustments:} Given disjoint sets $X,Y$ and element $i\not\in X\cup Y$, from Lemma 5 in \cite{desir} we have,
\[r^{X\cup Y}_i=r^X_i -\sum_{j\in Y} P_i(j\prec X\cup Y) r^X_j. \]
\item  Let $r_0=0$. From Lemma 7 in \cite{desir}, we have 
\begin{equation}\nonumber
R(i\mid Y)\geq 0,\quad \text{if $r_i\geq \max_{j\in Y\cup\{0\}} r_j$, $i\not\in Y$}.
\end{equation}
\item Let $S$ be an unconstrained optimal assortment. Then,
\[r^A_j\geq 0,\quad \forall A\subseteq S,\, j\in S.\]
\end{enumerate}  
Property (iv) is not shown directly in \cite{desir}, so we give a proof before proceeding.  By optimality of $S$, the set $B:=S\backslash A$ is an optimal unconstrained assortment in the $A$-adjusted markov chain. Suppose there exists $j\in B$ such that $r^A_j<0$. Then, we have a contradiction,
\[R^A(B)=\sum_{i\in B} P(i\prec S)r^A_i\, <\, \sum_{i\in B\backslash\{j\}} P(i\prec S\backslash \{j\})r^A_i=R^A(B\backslash\{j\}), \]
where the inequality uses $P(i\prec S)\leq P(i\prec S\backslash\{j\})$, which follows from substitutability of Markov model. 

Now, we show \eqref{prop2} and \eqref{prop1}. Let $S$ be a \emph{maximal} optimal unconstrained assortment on the ground set $N$, i.e., $R(i\mid S)<0$ for every $i\not\in S$. Consider sets $B\subseteq A\subseteq S$ and arbitrary set $C$.

\smallskip

\noindent \emph{Proving \eqref{prop1}:} $R(C\mid A)\leq R(C\mid B)$. Equivalently, we wish to show that,
\[\sum_{i\in C} P(i\prec A\cup C) r^{A}_i\leq \sum_{i\in C} P(i\prec B\cup C) r^{B}_i.  \]
From substitutability, we have $P(i\prec A\cup C)
\leq P(i\prec B\cup C)$ for every $i\in C$. Hence, it suffices to show $r^A_i\leq r^B_i,\, \forall i\in C$. Consider the $B$-adjusted markov chain. From property (iv), we have $r^B_j\geq 0$ for every $j\in S$.  Using property (ii),
\begin{eqnarray*}
r^A_i &= & r^B_i - \sum_{j\in A\backslash B}P_{i}(j\prec A) r^B_j\, \leq\, r^B_i.
\end{eqnarray*}

\smallskip

\noindent \emph{Proving \eqref{prop2}:} $R(A\mid C)\geq 0$. Given some element $e\in S$, we show that $R(e\mid C)\geq 0$ for every set $C\subseteq N\backslash \{e\}$. Applying this for every element in $A\backslash C$ proves the desired. From property (i), we have $R(e\mid C)\geq0$ if and only if $r^C_e\geq0$. Let $C_e=C+e$. Using property (ii), 
\begin{eqnarray}
r^C_e&= &r^{C\cap S}_e-\sum_{j\in C\backslash S}P_{e}(j\prec C) r^{C\cap S}_j\nonumber\\
&= & r^{C\cap S}_e -\sum_{j\in C\backslash S}P_{e}(j\prec C) \left(r^{C\cap S}_j- \sum_{i\in S\backslash C_e }P_{j}(i\prec S\backslash \{e\}) r^{C\cap S}_i\right)\nonumber \\
&& -\sum_{j\in C\backslash S}\sum_{i\in S\backslash C_e}P_{e}(j\prec C) P_{j}(i\prec S\backslash \{e\})r^{C\cap S}_i \nonumber\\
&\geq & r^{C\cap S}_e-\sum_{i\in S\backslash C_e}P_e(i\prec S\backslash\{e\})r^{C\cap S}_i - \sum_{j\in C\backslash S}P_{e}(j\prec C) \left(r^{C\cap S}_j- \sum_{i\in S\backslash C_e}P_{j}(i\prec S\backslash \{e\})r^{C\cap S}_i \right) \nonumber\\
&= &r^{S\backslash\{e\}}_e -\sum_{j\in C\backslash S} P_{e}(j\prec C) r^{S\backslash\{e\}}_j \label{last}
\end{eqnarray}
We note that the inequality in derivation above uses $r^{C\cap S}_i\geq0 $ for every $i\in S$.
Now, we claim that,
\[r^{S\backslash\{e\}}_e\geq r^{S\backslash\{e\}}_j,\quad  \forall j\in C\backslash S.\] 
Substituting this into \eqref{last} then gives us the desired. 
For the sake of contradiction, let there be a $j\in C\backslash S$ such that $r^{S\backslash\{e\}}_j>r^{S\backslash\{e\}}_e$. From property (i), recall that $r^{S\backslash\{e\}}_j=0$ for every $j\in S\backslash\{e\}$. Thus, we have $r^{S\backslash\{e\}}_j> \max_{i\in S}r^{S\backslash \{e\}}_i$. Applying property (iii) on the $S\backslash\{e\}$-adjusted markov chain, 
\[R^{S\backslash\{e\}}(S+j)\geq R^{S\backslash\{e\}}(S).\]
Thus, $R(S\cup\{j\})\geq R(S)$, contradicting the maximality of optimal assortment $S$.
%
This completes the proof of Theorem \ref{compatible}.
\hfill\Halmos
\end{proof}

\begin{proof}{Proof of Corollary \ref{markovcoro}.} Given a ground set $N$, discrete price ladder $\{p_1,\dots,p_r\}$, let $N_P=\{(i,p_j)\mid i\in N, p_j\in P\}$ denote an expanded universe of products. Given a Markov choice model over $N_P$, we can solve the joint pricing and assortment optimization problem by solving an instance of matroid constrained assortment optimization on ground set $N_P$. Consider a partition $\{R_1,\cdots,R_n\}$ of $N_P$, where $R_i=\{(i,p_j)\mid p_j\in P\},\, \forall i\in N$. Then, the partition matroid where a set $S$ is independent if and only if, $\forall (i,p_j)\in S$ we have $(i,p_k)\not \in S,\, \forall p_k\neq p_j$, enforces that each product can be included with at most one price. In fact, combining a cardinality constraint in addition to this partition matroid is still a matroid constraint. This completes the proof. 
\hfill\Halmos \end{proof}

\section{Missing details from Proof of Theorem \ref{compatible}}\label{appx:missing}

\subsection{Proof of Lemma \ref{interleafextend}}\label{appx:interextend}
\begin{repeatlemma}[Lemma \ref{interleafextend}.]
Consider a monotone subadditive function $f$ on ground set $N$ and $\{N_i,M_i\}_{i\in[p+1]}$ piece submodular order $\pi$. For every set $A$ that has an interleaved partitioned $\{O_\ell,E_\ell\}_{\ell\in[m]}$ such that $E(m) \subseteq \cup_{i\in[p+1]} M_i$, we have
\[f(A)\leq f(E(m)) +  \sum_{\ell} f(O_\ell\mid E(\ell-1)). \]
\end{repeatlemma}
\begin{proof}{Proof.}
Recall, $E(j)=\cup_{\ell\leq j} E_\ell$ and $O(j)=\cup_{\ell\geq j}O_\ell$. Also, $O(m+1)=\phi$.	The following inequalities are crucial for the proof,
\begin{equation}\label{induce2}
f\left( O(\ell)\cup E(m)\right)\leq f\left(O_{\ell} \mid E(\ell-1) \right)+ f\left( O(\ell+1) \cup E(m) \right) \quad \forall\, \ell\in[m].
\end{equation}
Summing up these inequalities for $\ell\in [m]$ we have,
\begin{eqnarray*}
\sum_{\ell\in[m]}f\left( O(\ell)\cup E(m)\right)-\sum_{\ell\in[m]}f\left( O(\ell+1) \cup E(m)\right)&&\leq \sum_{\ell\in[m]}f\left(O_{\ell} \mid \cup E(\ell-1) \right),\\
f\left( O(1)\cup E(m)\right)-f\left( E(m)\right) &&\leq \sum_{\ell\in[m]}f\left(O_{\ell} \mid \cup E(\ell-1) \right),\\
f(A)-f\left( E(m)\right) &&\leq \sum_{\ell\in[m]}f\left(O_{\ell} \mid \cup E(\ell-1) \right).
\end{eqnarray*} 

It remains to show \eqref{induce2}. For any $\ell\in[m]$, the sets $B:= E(\ell-1)$ and $A:= E(\ell-1)\cup O_{\ell}$ are proper sets and $B\subseteq A$. Further, the set $C:=O(\ell+1)\cup E(m) \backslash B$ lies entirely to the right of $A$.  Thus, from the piece-wise submodular order property, we have 
$f\left(C \mid A\right) \leq f(C\mid B)$ and

\begin{eqnarray*}
f\left( O (\ell)\cup E(m)\right)&= &f\left(A \right)+ f\left(C \mid A \right)\\
&\leq &f\left(B \right)+ f\left(O_{\ell} \mid B\right) +f\left(C \mid B\right)\\
&= &f\left(O_{\ell} \mid B\right) + f\left(B\cup C\right)= f\left(O_{\ell} \mid  E(\ell-1) \right) + f\left(O (\ell+1)\cup E(m) \right).
\end{eqnarray*}	
\hfill\Halmos \end{proof}

\subsection{Proof of Lemma \ref{forEextend}}\label{appx:forEextend}

\begin{repeatlemma}[Lemma \ref{forEextend}.]
Consider a monotone subadditive function $f$ on ground set $N$ and $\{N_i,M_i\}_{i\in[p+1]}$ piece submodular order $\pi$. We have that $f$ is $\pi$-submodular ordered on the restricted ground set $\cup_{i\in[p+1]} M_i$. 
Thus, given a set $A\subseteq \cup_{i\in[p+1]} M_i$, with interleaved partitioned $\{O_\ell,E_\ell\}_{\ell\in[m]}$ and sets $L_\ell=\{e\in A\mid \pi(e)<l_{\pi}(E_\ell)\},\, \forall \ell\in[m]$, we have 
\begin{enumerate}[(i)]
\item $f(A)\leq f\left(E(m)\right) + \sum_{\ell\in[m]} f\left(O_\ell \mid E(\ell-1)\cup O(1)\backslash O(\ell) \right),$
\item $\sum_{\ell\in[m]} f\left(E_\ell \mid L_{\ell}  \right) \leq f\left(E(m) \right).$ 
\end{enumerate}
\end{repeatlemma}
\begin{proof}{Proof.}
It suffices to show that $f$ is $\pi$-submodular ordered on $\cup_{i\in[p+1]} M_i$. The rest of the lemma the follows from Corollary \ref{forE}. On ground set $M_0=\cup_{i\in[p+1]} M_i$, consider arbitrary sets $B\subseteq A$ and a set $C$ to the right of $A$. Observe that $A$ is a proper set on the original ground set $N$. Therefore, using the piece-wise submodular order property on $N$ gives us the desired.
\hfill\Halmos \end{proof}

\subsection{Proof of Lemma \ref{pieceextend}}\label{appx:pieceextend}
\begin{repeatlemma}[Lemma \ref{pieceextend}.]
Consider an instance of constrained assortment optimization for a compatible choice model $\phi$. Suppose we execute Algorithm \ref{framework} with $\mathcal{A}$ given by an appropriate algorithm out of Algorithms \ref{calg}, \ref{bcalg}, \ref{5balg} and \ref{malg}. If for every parameter setting $\gamma\in \Gamma_\mathcal{A}$, the order $\pi$ and sets $\{N_i,M_i\}_{i\in[p+1]}$ generated by Algorithm \ref{framework} are such that $\pi$ is a $\{N_i,M_i\}_{i\in[p+1]}$ piece submodular order then, Algorithm \ref{framework} retains the guarantee of $\mathcal{A}$ for submodular order maximization.
\end{repeatlemma}
\begin{proof}{Proof.}
We consider each type of constraint (and the corresponding algorithms) separately. The analysis in each case mimics the analysis for submodular order functions. We omit $\phi$ from subscripts and denote $f_\phi$ as simply $f$. Recall that $f$ is monotone subadditive for substitutable choice models and compatible choice models are substitutable by definition. As shown in Part II of the analysis of Algorithm \ref{framework} (see Section \ref{sec:proofofcompat}), for compatible choice models the function $f$ exhibits piece-wise submodular order. For a given $\mathcal{A}_\gamma$, this order is characterized by sets $\{N_i,M_i\}_{i\in[p+1]}$ and ordering $\pi$ as defined in Section \ref{sec:proofofcompat}. 
\smallskip

\subsubsection*{Cardinality constraint.} $\mathcal{A}=$ Algorithm \ref{calg} and parameter $\gamma$ corresponds to threshold value $\tau$ in the algorithm. Let \opt\ denote both the optimal solution and function value. Notice that for $\opt\leq 2\max_{e\in N} f(\{e\})$, picking the largest value of $\tau$ gives a solution $S_{\gamma}$ such that $f(S_{\gamma})\geq (1-\epsilon)\,0.5\, \opt$. So from here on, let $\opt> 2\max_{e\in N} f(\{e\})$. 

Now, fix $\gamma$ such that $\tau\in [(1-\epsilon)\frac{\opt}{2k} ,\frac{\opt}{2k}]$ (such a setting exists for $\opt> 2\max_{e\in N} f(\{e\})$). In the following, we drop $\gamma$ in the notation for convenience. 
Let $\{N_i,M_i\}_{i\in[p+1]}$ and $\pi$ represent the resulting piece wise order generated. Recall that $S=\cup_{i\in[p]}M_i$. We show that $f(S)\geq (1-\epsilon)\, 0.5\, \opt$. 

Let $k'$ denote the cardinality of set $S$. Then, by definition of Threshold Add, \[f(S)\geq k'\tau.\]
When $k'=k$ this gives us $f(S)\geq (1-\epsilon)\, 0.5\, \opt$. So let $k'<k$.  
From monotonicity of the function we have, $\opt\leq f(\opt \cup S)$. So consider the union $\opt\cup S$ and its interleaved partition \[\{o_1,\{s_1\},o_2,\{s_2\},\cdots,\{s_{k'}\},o_{k'+1}\},\] where 
every element of $S$ is a singleton $s_j$ for some $j$ and every element of $\opt\backslash S$ is a singleton $o_{j}$ for some $j$. 
Note that $\{s_1,\cdots,s_\ell\}\subseteq \cup_{i\in[p]} M_i$ for every $\ell\in[k']$. So we apply Lemma \ref{interleafextend} on $\opt \cup S$ with $E_{j}:=s_j$ for $j\in[k']$, to obtain
\begin{eqnarray}
\opt&\leq &f\left(S\right) + \sum_{\ell\in[k']} f\left (o_\ell\mid E(\ell-1) \right). \nonumber
\end{eqnarray}
By definition of Threshold Add, 
we have $	f\left (o_{\ell}\mid E(\ell-1) \right) \leq \tau$, for every $\ell\geq 1$. 
Plugging this into the above inequality and using the upper bound $\tau\leq 0.5\opt/k$ we get,
\begin{eqnarray*}
\opt 
\leq f(S) + k\tau \leq f(S)+0.5\, \opt.
\end{eqnarray*}

\smallskip

\subsubsection*{Budget constraint.}
$\mathcal{A}=$ Algorithm \ref{5balg} (the case of Algorithm \ref{bcalg} is similar to cardinality constraint).
Parameter $\gamma$ determines set $X$ and threshold $\tau$. We focus on $X\subseteq \opt$ that contains the $1/\epsilon$ largest budget elements in \opt\ (or all of \opt\ if its cardinality is small). Given this $X$, the budget $b_e$ required by any $e\in \opt\backslash X$ is strictly smaller than $\epsilon B$, otherwise $b(X)>B$. Let $\tau\in \left[(1-\epsilon)\frac{\opt}{2B},\frac{\opt}{2B}\right]$. 

For the above setting of $\gamma$, let $\{N_i,M_i\}_{i\in[p+1]}$ and $\pi$ represent the piece wise submodular order generated. Let $S\cup R\subseteq \cup_{i\in[p]}M_i$ denote the output of $\mathcal{A}_\gamma$ where $S$ is the feasible set generated and $R$ is the set of elements (if any) removed by Final Add subroutine. Consider the following cases based on how $\mathcal{A}_\gamma$ terminates.
\smallskip

\noindent \textbf{Case I:} Final Add is not invoked, i.e., $R=\emptyset$. 
In this case elements are not removed after they are chosen and 
every element not in $S$ fails the threshold requirement. Similar to the case of cardinality constraint, we have an interleaved partition 
\[\{o_1,\{s_1\},o_2,\{s_2\},\cdots,\{s_{k}\},o_{k+1}\},\] where 
every element of $S$ is a singleton $s_j$ for some $j\in[k]$ and every element of $\opt\backslash S$ is a singleton $o_{j}$ for some $j\in[k+1]$. Using Lemma \ref{interleafextend} with $E_{j}:=s_j,\, \forall j\in[k]$,
\[\opt \leq f(S_i) + \sum_{j=1}^{k_i+1} f\left(o_{j} \mid E(j-1)\right).\]
Since $f\left(o_{j} \mid E(j-1)\right)\leq \tau,\, \forall j\in[k+1]$, we have $\alg\geq f(S)\geq  0.5\, \opt$.
\smallskip


\noindent \textbf{Case II:} Final Add is invoked for some element $j$. We claim that the final output $S$ is such that (i) $B\geq b(S)\geq (1-\epsilon)\, B$ and (ii) $f(S)\geq \tau b(S)$. Using (i) and (ii), we have for $\epsilon\in[0,1]$,
\[f(S)\geq 0.5\,(1-\epsilon)^2\, \opt\geq (0.5-\epsilon)\, \opt.\]
It remains to show (i) and (ii). Let $S_{in}$ denote the set that is input to Final Add and let $\hat{X}=\{e \mid b_e\geq \epsilon B, e\in S_{in}+\{j\}\}$. 
Observe that $\hat{X} \subseteq X$, since every element outside $X$ with budget exceeding $\epsilon B$ is discarded in the beginning. Therefore, $b(\hat{X})\leq B$ and the set $S$ returned by Final Add is feasible (and contains $\hat{X}$).

To see (i), let $t$ denote the last element removed from $S_{in}$ by Final Add. We have, $b_t<\epsilon B$ and $b_t+b(S)>B$. Thus, $b(S)\geq B-\epsilon B$. 

To show (ii), recall that $f$ is $\pi$ submodular ordered on the ground set $M_0:=\cup_{i\in[p]} M_i$. So we use Lemma \ref{forEextend} (ii) on set $S_{in}\subset S\cup R \subseteq M_0$. For simplicity, let $\{1,2,\cdots,s\}$ denote the elements of $S$ in submodular order. Recall that $S\subset S_{in}$. 
Using Lemma \ref{forEextend} (ii) with $A=S_{in}$ and sets $E_k=\{k\}$ for $k\in[s]$, we have
\[\sum_{k\in[s]} f(k\mid L_k)\leq f(S),\]
where $L_k$ denotes the set of all elements in $S_{in}$ chosen by $\mathcal{A}_\gamma$ prior to $k$. From the threshold requirement, we have $f(k\mid L_k)\geq \tau b_k,\, \forall k\in[s]$. Thus, $f(S)\geq \tau b(S)$. 
\smallskip

\subsubsection*{Matroid constraint.}
$\mathcal{A}=$ Algorithm \ref{malg} and there are no parameters $\gamma$. Let $S_j$ denote the set $S$ maintained in the algorithm at the beginning of iteration $j$. Similarly, let $R_j$ denote the set $R$ at the beginning of iteration $j$. Observe that $R_j$ is the set of all elements that were chosen and later swapped out, prior to iteration $j$. Therefore, the set $S_j\cup R_j$ grows monotonically and includes all elements selected by the algorithm prior to $j$. We use $S$ and $R$ to denote the final sets when the algorithm terminates. The sets $\{N_i,M_i\}_{i\in[p+1]}$ and order $\pi$ denote the piece-wise order generated. Observe that $S\cup R=\cup_{i\in[p+1]}M_i$.

Let $\opt$ denote the optimal set (and value). Similarly, we use \alg\ to denote both the value $f(S)$ and the set $S$ output by the algorithm. By monotonicity, $f(\alg \cup \opt)\geq \opt$. 
We will ``essentially" show that $f(\opt \mid \alg)\leq 3\alg$. The main elements that contribute to this upper bound are as discussed in the proof of Theorem \ref{resmatr} so we jump directly to the analysis.
%
%

W.l.o.g., ignore all elements in $N\backslash (\opt\cup S\cup R)$. So let $\hat{N}=\opt\cup S\cup R$ denote our ground set and re-index elements in $\hat{N}$ from $1$ to $|\hat{N}|$ (maintaining submodular order). Similarly, we re-index the sets $\{S_j\}$ so that they continue to denote the set maintained by the algorithm when element $j\in \hat{N}$ is first parsed, for every $j\in\{1,2,\cdots,|\hat{N}|\}$. From monotonicity, $\opt\leq f(\hat{N})$.

Consider an interleaved partition $\{o_\ell,e_\ell\}$ of $\hat{N}$ such that every element in $S\cup R$ is given by singleton $e_\ell$ for some $\ell\in[m]$ and every element in 
$\opt\backslash (S\cup R)$ is given by singleton $o_\ell$ for some $\ell\in[m]$. 
Using Lemma \ref{interleafextend},
\begin{eqnarray*}
\opt
\leq f\left(S\cup R\right) +\sum_{j\in \opt\backslash (S\cup R)} f\left(j \mid S_j\cup R_j\right),
\end{eqnarray*}
where $E(\ell)=\{e_1,\cdots,e_\ell\},\, \forall \ell\in[m]$ and $E(m)=S\cup R$. Next, we upper bound $f(S\cup R)$ in terms of \alg. Consider an interleaved partition $\{\bar{o}_\ell,\bar{e}_\ell\}_{\ell\in[\bar{m}]}$ of $S\cup R$ such that elements of $S$ are given by singletons $\bar{e}_\ell$ and singletons $\bar{o}_\ell$ represent elements of $R$. Since $S\cup R= \cup_{i\in[p+1]} M_i$, mwe apply Lemma \ref{forEextend} (i) to obtain, 
\begin{eqnarray*}
f\left(S\cup R\right)\, \leq\, \alg + \sum_{\ell\in[\bar{m}]} f\left(\bar{o}_\ell \mid \{\bar{o}_1,\bar{e}_1,\cdots,\bar{o}_{\ell-1},\bar{e}_{\ell-1}\} \right)\, = \,\alg+ \sum_{j\in R} f\left(j\mid S_j\cup R_j \right).
\end{eqnarray*}
In the following, we upper bound $\sum_{j\in R} f\left(j\mid S_j\cup R_j \right)$ by \alg\ and $\sum_{j\in \opt\backslash (S\cup R)} f\left(j \mid S_j\cup R_j\right)$ by $2\alg$. This proves the main claim. 
\smallskip

\noindent \textbf{Part I:} $\sum_{j\in R} f\left(j\mid S_j\cup R_j \right)\leq \alg$. 

\noindent Consider an element $j_1\in R$ that was added to $S_j$ without swapping out any element. Since $j_1\in R$, there exists an element $j_2$ that replaced $j$. Inductively, for $t\geq 2$, let $j_t$ denote the $t$th element in the chain of swaps $j_1\to j_2\to\cdots\to j_{t}$. The chain terminates at an element in $S$ and we call this a \emph{swap chain}. From the swap criteria in Algorithm \ref{malg}, we have
\[\sum _{\tau\in[t-1]} f\left(j_{\tau}\mid S_{j_\tau}\cup R_{j_\tau}\right)\leq f\left(j_t\mid R_{j_t}\cup S_{j_t}\right).\] 
Every element in $R$ is part of a unique swap chain 
and each chain has a unique terminal element in $S$. Therefore,
\[\sum_{j\in R} f\left(j\mid S_j\cup R_j \right)\leq \sum_{i\in S} f\left(i\mid S_i\cup R_i \right). \]
Applying Lemma \ref{forEextend} (ii) with $A=S\cup R$ and sets $E_i=\{i\}$ for $i\in S$, we have
\begin{equation}\label{coroE2}
\sum_{i\in S} f\left(i\mid S_i\cup R_i \right) \leq \alg. 
\end{equation} 
\smallskip

\noindent \textbf{Part II:} $ \sum_{j\in \opt\backslash (S\cup R)} f\left(j \mid S_j\cup R_j\right)\leq 2\alg$. 

\noindent From Lemma \ref{maximal} we have that the rank $r\left(\opt\right)\leq r\left(\alg\right)$. Suppose there exists an injection $\phi$ from elements in $\opt\backslash (S\cup R)$ to elements in $S$ such that
\[f\left(j \mid S_j\cup R_j\right) \leq 2 f\left(\phi(j) \mid S_{\phi(j)}\cup R_{\phi(j)} \right), \quad \forall j\in\opt\backslash (S\cup R).\] 
Summing up these inequalities and using \eqref{coroE2}, we are done. It remains to show that $\phi$ exists. We do this via a graphical construction inspired by \cite{chekuristream}.

Consider a graph $G$ with vertices given by $\hat{N}$. In order to define the edges recall that every $j\in\opt\backslash (S\cup R)$ is rejected by the algorithm on parsing. Thus, we have a unique circuit $C_j$ where $C_j \backslash \{j\}\subseteq S_j$. 
For our first set of edges we make a directed edge from $j$ to every element in $C_j\backslash\{j\}$ and we do this for all $j\in\opt\backslash (S\cup R)$. Next, for every $j\in R$, let $C_j$ represent the chain that causes $j$ to be swapped out in the algorithm. We create a directed edge from $j$ to every element in $C_j\backslash\{j\}$. Graph $G$ has the following properties,
\begin{enumerate}[(a)]
\item The elements of  $\opt\backslash (S\cup R)$ are \emph{source} vertices with no incoming edges. Elements of $S$ are \emph{sinks} with no outgoing edges.
\item The neighbors of every node in $G$ form a circuit with the node.
\item Given arbitrary vertex $j\in\opt\backslash (S\cup R)$, for every $i$ reachable from $j$ we claim that 
\[f(j\mid S_j\cup R_j)\leq v_i,\]
where $v_i$ corresponds to the value defined in Algorithm \ref{malg}. For neighbors of $j$ the claim follows directly from the swap criterion. Also, for any two neighboring vertices $i',i'' \in S\cup R$, we have $v_{i'}\leq v_{i''}$. The general claim follows by using these inequalities for every edge on the path from $j$ to $i$. 
\end{enumerate}
Using (a) and (b) we apply Lemma \ref{repeat} to obtain an injection $\phi$ from $\opt\backslash (S\cup R)$ to $S$ such that for every $j\in\opt\backslash (S\cup R)$ there exists a path to $\phi(j)\in S$. Then, from (c) we have, $f(j\mid S_j\cup R_j)\leq v_{\phi(j)}$. Recall the notion of a swap chain defined in Part I and let $W(\phi(j))$ denote the set of all preceding elements of the swap chain that terminates at $\phi(j)$. By definition of $v_{\phi(j)}$ and the swap criterion, \[v_{\phi(j)}=f\left(\phi(j)\mid S_{\phi(j)}\cup R_{\phi(j)}\right)+\sum_{i\in W(\phi(j))}f\left(i\mid S_{i}\cup R_{i}\right) \leq 2f\left(\phi(j)\mid S_{\phi(j)}\cup R_{\phi(j)}\right).\] 
\hfill\Halmos \end{proof}
\section{Streaming Maximization of Monotone Submodular Functions}\label{sec:streamproof}
{\color{black}
\begin{repeattheorem}[Theorem \ref{stream}.]
For constrained maximization of a monotone submodular function $f$,
\begin{enumerate} [(i)]
\item Algorithm \ref{calg} gives a $(1-\epsilon)\, 0.5$ approximation in the streaming setting.
\item Algorithm \ref{bcalg} gives a $\frac{1}{3}-\epsilon$ approximation algorithm in the streaming setting.
\item Algorithm \ref{modmalg} gives a $0.25$ approximation in the streaming setting.
\end{enumerate}  
\end{repeattheorem}
\begin{proof}{Proof.}
Let $\{1,2,\cdots,n\}$ denote the streaming order.	This is a submodular order for $f$, since every permutation of the ground set is a submodular order ($f$ is submodular). It remains to show that Algorithms \ref{calg}, \ref{bcalg}, and \ref{malg} are streaming algorithms, i.e., they parse the ground set just once and require $\tilde{O}(k)$ memory, where $k$ is the size of a maximal feasible solution and $\tilde{O}$ ignores factors that are polynomials of $\log k$. 

To see that Algorithm \ref{calg} is a streaming algorithm, observe that for any given value of $\epsilon>0$, we can implement Algorithm \ref{calg} such that it parses the ground set exactly once for each value of the threshold $\tau$. Each pass over the ground set requires a memory of $k$ (the cardinality parameter). Alternatively, using $O(k\log k/\epsilon)$ memory we can simultaneously maintain a candidate solution for each value of $\tau$ and perform just one pass over the ground set. Similarly, Algorithm \ref{bcalg} is also a streaming algorithm. Algorithm \ref{malg} passes over the ground set exactly once. However, the naive implementation of the algorithm stores the set $R$ which can be of size $O(n)$. For a function with strong submodular order, one can use Algorithm \ref{modmalg} that only uses marginals $f(j\mid S)$, i.e., we can ignore the set $R$ in the algorithm. This reduces the memory required for implementation to $O(d)$, as desired. Algorithm \ref{modmalg} is formally discussed and analyzed below.

\hfill\Halmos	\end{proof}

\begin{algorithm}[H]
\SetAlgoLined
\textbf{Input:} Independence system $\mathcal{I}$, $N$ indexed in submodular order\; 
\smallskip
Initialize $S=\emptyset$ and values $v_j=0$ for all $j\in[n]$\;
\For{$j\in\{1,2,\cdots,n\}$}{
\lIf {$S\cup\{j\}\in \mathcal{I}$} {initialize $v_j=f(j\mid S)$ and update $S\to S \cup \{j\}$}
\Else{ 
Find circuit $C$ in $S\cup\{j\}$ and define $i^*=\underset{i\in C\backslash\{j\}}{\arg\min}\, v_i$ 
and $v_C=v_{i^*}$\; 
\If  {$f(j|S)
>v_C$} { 
$v_j\to v_C + f(j\mid S)$\;
$S\to (S\backslash\{i^*\})\cup \{j\}$\;
}}}
\smallskip
\textbf{Output:} Independent set $S$
\caption{$\frac{1}{4}$ for Matroid Constraint (Efficient Version for Strong Submodular Order) } 
\label{modmalg}
\end{algorithm}
\begin{theorem}
Algorithm \ref{modmalg} is $0.25$ approximate with query complexity $O(nd)$ for the problem of maximizing a monotone subadditive function $f$ with a known strong submodular order, subject to a matroid constraint with matroid rank $d$.
\end{theorem}
\begin{proof}{Proof.} Algorithm \ref{modmalg} parses the ground set once. When parsing an element, the algorithm makes at most $d$ queries, resulting in total $nd$ queries.

To show the performance guarantee. Let $S_j$ denote the set $S$ maintained in the algorithm at the beginning of iteration $j$. Similarly, let  $R_j$ be the set of all elements that were chosen and later swapped out, prior to iteration $j$. Therefore, the set $S_j\cup R_j$ grows monotonically and includes all elements selected by the algorithm prior to $j$. We use $S$ and $R$ to denote the final sets when the algorithm terminates. 
Let $\opt$ denote the optimal set (and value) and \alg\ denote the algorithm output (and value). By monotonicity, $f(\alg \cup \opt)\geq \opt$. 
We will ``essentially" show that $f(\opt \mid \alg)\leq 3\alg$. The main elements that contribute to this upper bound are,
\smallskip

\noindent \textbf{Swap operations:} If an element $j$ is taken out of the set at iteration $t$, then the elements in \opt\ parsed between $j$ and $t$, but never included, may have larger marginal value after the swap. We show that the resulting total increase in marginal value is upper bounded by \alg. 
\smallskip

\noindent \textbf{Rejection of elements:} Any element that is rejected by the algorithm on parsing (due to insufficient marginal value for swapping) may be an element of \opt. Let $j$ denote the index (in submodular order) of such an element. 
Since $j$ is rejected we have that $S_j\cup \{j\}$ contains a circuit $C_j$ and the marginal $f(j\mid S_j)$ is upper bounded by $v_i$ for all $i\in C_j$. 
We show that the total marginal value $\sum_{j\in \opt\backslash (S\cup R)} f(j\mid S_j)$ of elements in $\opt$ rejected by the algorithm is at most 2\alg.
\smallskip

W.l.o.g., we ignore all elements in $N\backslash (\opt\cup S\cup R)$. So let $\hat{N}=\opt\cup S\cup R$ denote our ground set and re-index elements in $\hat{N}$ from $1$ to $|\hat{N}|$ (maintaining submodular order). Similarly, we re-index the sets $\{S_j\}$ so that they continue to denote the set maintained by the algorithm when element $j\in \hat{N}$ is first parsed, for every $j\in\{1,2,\cdots,|\hat{N}|\}$. From monotonicity, $\opt\leq f(\hat{N})$.

Consider an interleaved partition $\{O_\ell,E_\ell\}$ of $\hat{N}$ such that $E(m)=\cup_{\ell\in[m]}E_\ell:= S\cup R$ and $O(1)=\cup_{\ell_\in[m]}O_\ell:= \opt\backslash (S\cup R)$. 
Using Corollary \ref{interleaf},
\begin{eqnarray*}
\opt&\leq &f\left(E(m)\right) +\sum_{\ell\in[m]} f\left(O_\ell \mid E(\ell-1)\right),\\
&\leq &f\left(E(m)\right) +\sum_{j\in \opt\backslash (S\cup R)} f\left(j \mid S_j\cup R_j\right),
\end{eqnarray*}
where the second inequality follows from weak submodular order. Next, we apply Corollary \ref{forE}(i) to upper bound $f(E(m))$ in terms of \alg. Consider an interleaved partition $\{\bar{O}_\ell,\bar{E}_\ell\}_{\ell\in[\bar{m}]}$ of $E(m)$ such that $\bar{E}(\bar{m})=S$ and $\bar{O}(1)=R$. We have,
\begin{eqnarray*}
f\left(E(m)\right) &\leq &\alg + \sum_{\ell\in[\bar{m}]} f\left(\bar{O}_\ell \mid \bar{E}(\ell-1) \cup \bar{O}(1)\backslash \bar{O}(\ell) \right),\\
&= &\alg+ \sum_{j\in R} f\left(j\mid S_j\cup R_j \right).
\end{eqnarray*}
Using the fact that $f$ has a strong submodular order, we have,
\[f\left(j\mid S_j\cup R_j \right)\leq f\left(j\mid S_j \right).\]
Therefore, to prove the lemma it suffices to upper bound $\sum_{j\in R} f\left(j\mid S_j\right)$ by \alg\ and $\sum_{j\in \opt\backslash (S\cup R)} f\left(j \mid S_j\right)$ by $2\alg$. 
\smallskip

\noindent \textbf{Part I:} $\sum_{j\in R} f\left(j\mid S_j \right)\leq  \alg$. 

\noindent Consider an element $j_1\in R$ that was added to $S_j$ without removing any element. Since $j_1\in R$, there exists an element $j_2$ that replaced $j$. Inductively, for $t\geq 2$, let $j_t$ denote the $t$th element in the chain of swaps $j_1\to j_2\to\cdots\to j_{t}$. The chain terminates at an element in $S$ and we call this a \emph{swap chain}. From the swap criteria in Algorithm \ref{modmalg}, we have
\[\sum _{\tau\in[t-1]} f\left(j_{\tau}\mid S_{j_\tau}\right)\leq f\left(j_t\mid  S_{j_t}\right).\] 
Every element in $R$ is part of a unique swap chain 
and each chain has a unique terminal element in $S$. Therefore,
\[\sum_{j\in R} f\left(j\mid S_j \right)\leq \sum_{i\in S} f\left(i\mid S_i \right). \]
Let $S=\{s_1,s_2,\cdots, s_q\}$. Then, we have
\begin{equation}\label{modcoroE}
\alg=	\sum_{i\in S} f\left(i\mid \{s_1,\cdots, s_{i-1}\}\right)\geq \sum_{i\in S} f\left(i\mid S_i \right),
\end{equation} 
here the inequality uses the strong submodular order property of $f$. 
\smallskip

\noindent \textbf{Part II:} $ \sum_{j\in \opt\backslash (S\cup R)} f\left(j \mid S_j\right)\leq 2\, \alg$. 

\noindent From Lemma \ref{maximal} we have $r\left(\opt\right)\leq r\left(\alg\right)$. Suppose that there exists an injection $\phi$ from elements in $\opt\backslash (S\cup R)$ to elements in $S$ such that
\[f\left(j \mid S_j\right) \leq 2 f\left(\phi(j) \mid S_{\phi(j)} \right), \quad \forall j\in\opt\backslash (S\cup R).\] 
Summation of these inequalities and using \eqref{modcoroE}, we are done. It remains to show that $\phi$ exists. We do this via a graphical construction inspired by \cite{chekuristream}.

Consider a graph $G$ with vertices given by $\hat{N}$. In order to define the edges recall that every $j\in\opt\backslash (S\cup R)$ is rejected by the algorithm on parsing. Thus, we have a unique circuit $C_j$ where $C_j \backslash \{j\}\subseteq S_j$. 
For our first set of edges, we make a directed edge from $j$ to every element in $C_j\backslash\{j\}$ and we do this for all $j\in\opt\backslash (S\cup R)$. Next, for every $j\in R$, let $C_j$ represent the chain that causes $j$ to be swapped out in the algorithm. We create a directed edge from $j$ to every element in $C_j\backslash\{j\}$. Graph $G$ has the following properties,
\begin{enumerate}[(a)]
\item The elements of  $\opt\backslash (S\cup R)$ are \emph{source} vertices with no incoming edges. Elements of $S$ are \emph{sinks} with no outgoing edges.
\item The neighbors of every node in $G$ form a circuit with the node.
\item Given arbitrary vertex $j\in\opt\backslash (S\cup R)$, for every $i$ reachable from $j$ we claim that 
\[f(j\mid S_j)\leq v_i,\]
where $v_i$ corresponds to the value defined in Algorithm \ref{malg}. For neighbors of $j$ the claim follows directly from the swap criterion. Also, for any two neighboring vertices $i',i'' \in S\cup R$, we have $v_{i'}\leq v_{i''}$. The general claim follows by using these inequalities for every edge on the path from $j$ to $i$. 
\end{enumerate}
Using (a) and (b) we apply Lemma \ref{repeat} to obtain an injection $\phi$ from $\opt\backslash (S\cup R)$ to $S$ such that for every $j\in\opt\backslash (S\cup R)$ there exists a path to $\phi(j)\in S$. Then, from (c) we have, $f(j\mid S_j)\leq v_{\phi(j)}$. Recall the notion of a swap chain defined in Part I and let $W(\phi(j))$ denote the set of all preceding elements of the swap chain that terminates at $\phi(j)$. By definition of $v_{\phi(j)}$ and the swap criterion, 
\[v_{\phi(j)}=f\left(\phi(j)\mid S_{\phi(j)}\right)+\sum_{i\in W(\phi(j))}f\left(i\mid S_{i}\right) \leq 2f\left(\phi(j)\mid S_{\phi(j)}\right).\] 

\hfill\Halmos \end{proof}
}
\end{APPENDICES}
\end{document}